\documentclass[a4paper,10pt,reqno]{amsart}

\usepackage[T1]{fontenc}			
\usepackage[utf8]{inputenc}
\usepackage[italian, english]{babel}
\usepackage{amsmath, amsfonts, amssymb, amsthm, amscd, mathtools}
\usepackage{geometry}
\usepackage[square,numbers]{natbib}	
\usepackage{bbm, bm}
\usepackage{xspace} 
\usepackage{enumerate, enumitem}
\usepackage{hyperref}
\usepackage{xcolor}
\usepackage{comment}

\usepackage{thmtools}

\newcommand{\bbF}{{\ensuremath{\mathbb F}} }

\newcommand{\cB}{{\ensuremath{\mathcal B}} }
\newcommand{\cC}{{\ensuremath{\mathcal C}} }

\newcommand{\cF}{{\ensuremath{\mathcal F}} }

\newcommand{\cL}{{\ensuremath{\mathcal L}} }

\newcommand{\cP}{{\ensuremath{\mathcal P}} }

\newcommand{\cS}{{\ensuremath{\mathcal S}} }

\newcommand{\bfJ}{{\ensuremath{\mathbf J}} }

\newcommand{\dC}{{\ensuremath{\mathrm C}} }

\newcommand{\R}{\mathbb{R}}

\renewcommand{\P}{\mathbb{P}}
\newcommand{\E}{\mathbb{E}}

\newcommand{\ind}{\ensuremath{\mathbf{1}}}

\DeclarePairedDelimiterX{\inprod}[2]{\langle}{\rangle}{#1, #2}

\newcommand{\changeepsilon}
{
\let\temp\epsilon
\let\epsilon\varepsilon
\let\varepsilon\temp
}

\newcommand{\changetheta}
{
\let\temp\theta
\let\theta\vartheta
\let\vartheta\temp
}
\newcommand{\changephi}
{
\let\temp\phi
\let\phi\varphi
\let\varphi\temp
}

\declaretheorem[name=Theorem]{theorem}

\theoremstyle{plain}
\newtheorem*{theorem*}{Theorem}
\newtheorem{lemma}[theorem]{Lemma}
\newtheorem*{lemma*}{Lemma}
\newtheorem{proposition}[theorem]{Proposition}
\newtheorem*{proposition*}{Proposition}

\theoremstyle{definition}
\newtheorem{definition}{Definition}[section]
\newtheorem{assumption}{Assumption}[section]

\theoremstyle{remark}
\newtheorem{remark}{Remark}[section]
\newtheorem*{remark*}{Remark}

\definecolor{darkviolet}{rgb}{0.58, 0.0, 0.83}

\definecolor{gre}{rgb}{0.03,0.50,0.03}

\numberwithin{equation}{section}

\hypersetup{pdftitle={stationary_MFC_OU},pdfauthor={Federico Cannerozzi},hidelinks,%
pdfstartview={XYZ null null 1.4},colorlinks=true,linkcolor=blue,urlcolor=blue}

\newcommand{\mail}[1]{\href{mailto:#1}{\normalfont\texttt{#1}}}

\makeatletter
\def\@setthanks{\vspace{-\baselineskip}\def\thanks##1{\@par##1\@addpunct.}\thankses}
\makeatother

\title[Stationary Mean-Field Control of an Ornstein-Uhlenbeck process]{Stationary Mean-Field singular control of an Ornstein-Uhlenbeck process}
\author[Federico Cannerozzi]{Federico Cannerozzi\textsuperscript{\MakeLowercase{a},1}}
\thanks{\noindent \textsuperscript{a} Bielefeld University, Center for Mathematical Economics (IMW), Bielefeld (Germany).}
\thanks{ \textsuperscript{1} E-mail: \mail{federico.cannerozzi@uni-bielefeld.de}. }

\date{\today}

\begin{document}

\changephi
\changeepsilon
\allowdisplaybreaks

\begin{abstract}
Motivated by continuous-time optimal inventory management, we study a class of stationary mean-field control problems with singular controls. The dynamics are modeled by a mean-reverting Ornstein-Uhlenbeck process, and the performance criterion is given by a quadratic long-time average expected cost functional. The mean-field dependence is through the stationary mean of the controlled process itself, which enters the ergodic cost functional.
{\color{black} We relate optimal controls for the stationary mean-field control problem to equilibria of an associated stationary mean-field game. In particular, optimal controls induce equilibria of the associated game, while the converse holds true within the class of admissible controls satisfying an additional integrability requirement. Finally, we solve the stationary mean-field game explicitly, thereby obtaining an explicitly characterized optimal reflection policy for the stationary mean-field control problem within this restricted class.}
\end{abstract}

\maketitle


\smallskip

\noindent \textbf{AMS 2020:} 49N80, 91A15, 91A16, 90B05.

\smallskip

\section{Introduction}\label{sec:intro}

\subsection*{Problem and motivation}
In this paper, we investigate a stationary singular mean-field control (MFC) problem.
The state process of the agent is given by a singularly controlled mean-reverting Ornstein-Uhlenbeck process, where the control is a process of bounded variation that enters linearly the dynamics.
The control problem is two-sided, meaning that the control can be exercised both upward and downward.
The agent seeks to minimize a long-term cost functional, given by the long-time average of an expected cost.
The instantaneous cost, of quadratic type, depends on the state of the agent and on the stationary mean of the controlled process itself, casting the problem in the framework of mean-field control problems.
In the following, we refer to this optimization problem as stationary MFC problem. 

\smallskip
The stationary MFC problem can be motivated as a continuous-time optimal inventory management problem.
Brownian models for the inventory process are nowadays classical; see, e.g., the review \cite{perera2023inventory_models} and the references therein.
In our framework, the state process can be regarded as the net inventory process of a firm, which captures the difference between {\color{black} regular supply and customer demand}.
The net inventory process exhibits mean-reverting behavior, in line with literature (see, e.g., \cite{cadenillas2010optimal,hu_liu2018mean_reverting,liu_bensoussan2019ergodic_inventory}).
In this context, mean-reversion can be thought of as an effect of deterioration (see \cite{goyal2001recent,raafat1991survey}).
The firm controls its inventory level by a process of bounded variation, where the positive and negative parts of each singular control represent the total amount of increase/decrease of inventory up to time $t$.
The cost is of long-time average (ergodic) type, following a long history in the inventory management literature (among others, see \cite{calvia2025optimal,dai2013brownian_inventory,helmes2025long,helmes2017continuous,helmes2018weak_inventory,helmes2025inventory,liu_bensoussan2019ergodic_inventory,taksar1985average,yao2015optimal_brownian_inventory}).
Controls are exercised on the net inventory process to maintain the inventory at {\color{black} a} desired position, to minimize the long-term variance and to reach a target on the long-term average of the inventory process itself.
The stationary one-dimensional setting of the MFC problem, together with its linear-quadratic structure, allows for explicit characterizations of the optima and for comparative statics.
To the best of our knowledge, this is the first work that includes a mean-field control feature in an inventory problem, through terms depending on the stationary mean of the inventory process itself.

\subsection*{Our contribution}
We solve the stationary MFC problem by establishing a new connection with stationary mean-field games (MFGs).
{\color{black}
We relate optimal controls for the stationary MFC problem to equilibria of a new stationary MFG, which we refer to as the potential stationary MFG.
More precisely, we prove that any optimal control for the stationary MFC problem induces an equilibrium of the potential stationary MFG.
Conversely, equilibria satisfying the conditions of our verification result yield optimal controls for the stationary MFC problem within the class of admissible controls satisfying an additional integrability requirement.
}
The potential stationary MFG shares the same dynamics as the stationary MFC problem but features a modified cost functional that depends on two scalar interaction terms.
At equilibrium, we require that the first interaction term coincides with the stationary mean of the optimally controlled process, while the second interaction term matches the integral of the derivative of the instantaneous cost with respect to the interaction parameter, taken with respect to the stationary distribution of the optimally controlled process.
{\color{black}
To the best of our knowledge, this paper is the first to establish such a two-way relation between a stationary mean-field singular control problem and an associated potential stationary MFG, and to exploit it to obtain an explicit characterization of an optimal policy for the stationary MFC problem.
}

The introduction of this second interaction term, together with its associated consistency condition, is closely tied to the way we show that any solution to the stationary MFC problem induces a solution to the potential stationary MFG.
Specifically, we relate the stationary MFC problem to a two-step optimization procedure: one first optimizes over admissible controls satisfying a constraint on the stationary mean, and then optimizes over the mean-field parameter.
This two-step optimization procedure, already present in literature in \cite{cannerozzi2024cooperation,christensen2021competition,helmes2025long}, is used here in a new way to obtain a necessary condition for optimality for the MFC problem in terms of solutions to the potential stationary MFG.

We first solve the constrained optimization problem using a Lagrange multiplier approach. This allows us to solve an associated unconstrained optimization problem and then tune the Lagrange multiplier so that the resulting optimal control satisfies the constraint on the stationary mean.
This procedure turns out to be equivalent to solving a family of stationary MFGs with a consistency condition on the stationary mean, and then showing that one can choose the second interaction parameter so that the equilibrium satisfies the constraint on the stationary mean.

The unconstrained optimization problem is solved through a suitable connection with Dynkin games, a connection already observed in the context of ergodic control problems (see \cite{calvia2025optimal,karatzas1983class,taksar1985average}) and used in the study of one-dimensional stationary MFGs of singular controls with ergodic objective (see \cite{cannerozzi2024cooperation,cao2023stationary,dianetti2023ergodic}).
This connection, in particular, enables us to show that the optimal control is of barrier type: it prescribes the controlled process to be kept between two deterministic thresholds, which we prove to be Lipschitz functions of the interaction parameters, with explicit {\color{black} Lipschitz} constants.

Having solved the constrained optimization problem, we find that the Lagrange multiplier associated with the stationary mean of the optimal control process must satisfy the second consistency condition. This clarifies the role of the additional interaction term in the potential stationary MFG: it should be interpreted as the Lagrange multiplier corresponding to the optimal stationary mean in the second step of the optimization procedure.
Thus, while the first consistency condition can be viewed as an admissibility requirement, the second consistency condition should be interpreted as an optimality condition on the stationary mean.

{\color{black} To prove the sufficient condition, we first establish suitable first-order optimality conditions for controls satisfying an additional integrability requirement.
We then use these conditions to show that any equilibrium of the potential stationary MFG satisfying the required regularity conditions yields an optimal control for the stationary MFC problem within the class of admissible controls satisfying the same additional integrability requirement.}
The consistency condition on the second interaction term plays a crucial role in relating the modified cost functional of the potential MFG to the cost functional of the original stationary MFC problem.

Finally, we show that the potential stationary MFG admits a solution. By solving the finite-dimensional consistency conditions determining the stationary mean and the associated multiplier, we obtain a concrete optimal policy for the original stationary MFC problem {\color{black} within the class of admissible controls satisfying the additional integrability requirement}: the controlled process is minimally reflected between two constant boundaries, which depend on the stationary mean of the optimally controlled process itself.
{\color{black} Moreover, we prove uniqueness for the stationary MFG within the class of reflection policies, from which we deduce uniqueness for the stationary MFC in the same class.}

{\color{black} We emphasize that the potential stationary MFG formulation is not only a verification tool, but also the key step allowing for a tractable and semi-explicit characterization of the optimal policy. Indeed, the direct analysis of the stationary MFC problem is made intricate by the fact that the mean-field interaction is computed from the stationary distribution associated with each admissible singular control. By fixing the interaction parameters, we avoid directly studying variations of stationary distributions or developing a calculus on spaces of probability measures, at the price of subsequently imposing two scalar consistency conditions. The resulting fixed-parameter problem is a classical one-dimensional ergodic singular control problem which, in the present Ornstein-Uhlenbeck setting, can be solved through an auxiliary Dynkin game. This reveals that the optimizer is a two-sided reflection policy and provides a characterization of its reflection boundaries as functions of the interaction parameters.}



We conclude the paper with numerical illustrations of our findings, highlighting how the constructed optimal reflection policy and its stationary mean depend on key parameters such as the speed of mean reversion, volatility, and interaction strength.

\subsection*{Related literature}
This paper contributes to various streams of literature.
First, it contributes to the literature on mean-field control problems with ergodic cost. Contributions in this area are relatively scarce; see \cite{albeverio2022sicon,bao_tang2023ergodic_ctrl_mkv,bayraktar2024discrete,cannerozzi2024cooperation,fuhrman2025ergodicMKV}. In particular, in \cite{albeverio2022sicon} existence and uniqueness results are obtained for a class of ergodic MFC problems; moreover, the authors introduce an $N$-agent Markovian control problem for which the optimal control can be explicitly characterized, and show that its payoff converges to the value of the ergodic MFC problem.
\cite{bayraktar2024discrete} studies an ergodic MFC problem in a discrete-time framework, while \cite{bao_tang2023ergodic_ctrl_mkv,fuhrman2025ergodicMKV} establish existence and uniqueness of an optimal control by analyzing the Hamilton-Jacobi-Bellman equation, either on the Hilbert space $L^2$ of square-integrable random variables or on the Wasserstein space of square-integrable probability measures.
\cite{cannerozzi2024cooperation} studies a one-sided stationary MFC problem with an explicit geometric structure, which serves as a benchmark model for Pareto-efficient strategies in a game with strategic complementarity.
We also refer to \cite{bayraktar2025ergodicity,sun2024ergodic_tunrpike}, where the so-called turnpike property is investigated.
{\color{black} Finally, \cite{christensen2021competition,helmes2025long} are particularly related to our work, as they employ a Lagrange multiplier approach to compute solutions to MFC problems in the context of impulse control problems with stationary dependence on the interaction term.}
{\color{black} Outside the ergodic setting, \cite{guo_pham_wei2023itoformula} derives an It\^o's formula for flows of measures on semimartingales and applies it to McKean-Vlasov mixed regular-singular control problems, including a mean-variance singular control example.}

Secondly, we contribute to the literature on stationary MFGs of singular controls.
Stationary MFGs have a long history: contained already in the seminal paper \cite{lasry_lions}, they have been addressed in various forms. We refer to \cite{araposthatis2017mfgs_ergodic_cost,cirant2016stationary,dragoni2018ergodic_mfgs} for existence results, to  \cite{bardi2014sicon,feleqi2013derivation_ergodic_mfgs} for forward-convergence results, and to \cite{bardi2024SIMAN,cardaliaguet2012long_time_averages,cecchin2024turnpike} for the long-time convergence of finite-horizon MFGs to stationary MFGs.
Referring in particular to MFGs of singular controls, they have been addressed in a variety of papers:
\cite{cohen2024existence,denkert2024extended,dianetti2023unifying,fu2023extended,fu2017mean} consider abstract existence results, while explicit characterizations of the MFG equilibria can be found in \cite{aid2025stationary,campi2022mean,cao2023stationary,cao2022mfgs, dianetti2023ergodic,guo2019stochastic}.

Thirdly, and perhaps most importantly, we contribute to the literature on potential MFGs and their connection with MFC problems. It was already observed in \cite[Section 2.6]{lasry_lions} that, for the class of potential MFGs, one can introduce an optimization problem whose solutions coincide with the solutions of the MFG.
This connection has been exploited in a variety of works to solve MFGs via auxiliary optimization problems; see, for example, \cite{briani_cardaliaguet2018potential,cardaliaguet2015weak,cardaliaguet2015firstorder,cardaliaguet2015secondorder,graber2014AMOP,graber2021potential,orrieri2019variational} for the analytic approach to MFGs and \cite[Chapter 6.7.2]{librone_vol_I}, \cite{carmona2025conditional,graber2016LQ_exhaustible_resource} for the probabilistic approach.
{\color{black}
Our approach differs from the aforementioned works and, to the best of our knowledge, is not present in the literature, as our goal is to establish a two-way relation between a stationary MFC problem and an associated potential stationary MFG, and to use this relation to construct explicitly an optimal policy for the stationary MFC problem.
Specifically, we first derive a necessary condition for the stationary MFC problem in terms of equilibria of the potential stationary MFG, and then establish a converse sufficient condition within the class of admissible controls satisfying an additional integrability requirement.
}
In this sense, the connection with the potential MFG plays the role of both a first-order optimality condition, analogous to \cite{hofer2024optimal}, and a sufficient condition.
Moreover, our approach prescribes modifying the original cost functional so that solutions to the potential stationary MFG yield solutions to the stationary MFC problem {\color{black} within an integrability-restricted admissible class}.
Related ideas appear in \cite{alasseur2024mfg}, where a modified MFG is used to solve an MFC problem, in \cite{graber2016LQ_exhaustible_resource}, where assumptions under which MFG and MFC solutions coincide are discussed in a linear-quadratic framework, and in \cite{carmona2023nash}, where it is shown that MFGs with suitably modified cost functionals can achieve the same payoff as the original MFC problem.
{\color{black} We also refer to \cite{amato2026singular}, where the connection between mean-field singular control problems and auxiliary potential MFGs is developed in a finite-horizon framework with general measure dependence. Under suitable convexity assumptions, that work proves that solutions to an associated potential MFG yield solutions to the original MFC problem, and applies this result to a one-sided mean-field monotone follower problem. Differently from \cite{amato2026singular}, the present paper focuses on a stationary ergodic problem with two-sided singular controls and Ornstein-Uhlenbeck dynamics, establishes both directions of the connection between the stationary MFC problem and the potential stationary MFG, subject to the additional integrability condition required for the sufficient direction, and exploits this connection to obtain a semi-explicit characterization of the optimal stationary policy through two constant reflection boundaries.}

\subsection*{Organization of the paper}
The rest of the paper is organized as follows. Section \ref{sec:ctrl_problem} introduces the problem and states the main results. The necessary condition for optimality for the stationary MFC problem is established in Section \ref{sec:necessary}, while Section \ref{sec:sufficient} addresses both the sufficient condition and the existence of an equilibrium for the potential stationary MFG. Finally, Section \ref{sec:numerics} presents some numerical illustrations of our findings.

\section{Model and Main Results}\label{sec:ctrl_problem}
Let $(\Omega,\cF,\bbF = (\cF_t)_{t \geq 0},\P)$ be a complete filtered probability space capturing all the uncertainty of our setting.
Consider a real-valued $\bbF$-Brownian motion $W$ and let $\delta, \, \sigma > 0$.

\smallskip
Let $\cB$ be the set of $\bbF$-adapted processes of bounded variation with càdlàg paths which start from $0$ at time $0^{-}$.
We identify each process $\xi \in \cB$ with its positive and negative parts given by the Jordan decomposition, i.e. $\xi_t = \xi^+_t - \xi^-_t$, with $t \mapsto \xi^\pm_t$ $\bbF$-adapted, right-continuous, non-decreasing and so that $\xi^\pm_{0^{-}}=0$ $\P$-a.s.
Moreover, we assume that $\E[\vert \xi \vert _T] < \infty$ for any $T > 0$.
We call any process $\xi \in \cB$ a singular control.

\smallskip
For any {\color{black} $\xi \in \cB$}, the dynamics is given by
\begin{equation}\label{application:dynamics}
    d X^{\xi}_t =  - \delta X^{\xi}_tdt + \sigma dW_t + d\xi^+_t - d\xi^-_t, \quad X^{\xi}_{0^{-}}=x.
\end{equation}

Let $\cP(\R)$ be the set of probability measures over $\R$.
In order to define the stationary mean-field control problem, we restrict to those controls so that the process $X^{\xi}$ is ergodic, in the sense of the following definition:
\begin{definition}\label{def:ergodic_strategies}
We say that a strategy $\xi \in \cB$ is admissible for the stationary mean-field control problem if the process $(X^{\xi}_t)_{t \geq 0^{-}}$, with $X^{\xi}$ solution to \eqref{application:dynamics}, admits a unique limiting stationary distribution $p^{\xi}_\infty \in \cP(\R)$ with $\int_{\R} \vert x \vert p^{\xi}_\infty(dx) < \infty$.
We denote the set of admissible strategies for the stationary mean-field control problem by $\cB_{mfc}$, and we denote by $\theta^{\xi}$ the first moment of $p^{\xi}_\infty$, i.e. $\theta^{\xi} = \int_{\R} x \, p^{\xi}_\infty(dx)$.
\end{definition}

For $\xi \in \cB_{mfc}$, define
\begin{equation}\label{eq:cost:ergodic}
    \bfJ^{mfc}(\xi) \coloneqq \varlimsup_{T \uparrow \infty}\frac{1}{T}\E\left[\int_0^T \frac{1}{2}\left( \rho(X^{\xi}_t - \bar{x})^2 + \phi(X^{\xi}_t - \theta^{\xi})^2 + \psi(\theta^\xi - \bar{\theta})^2 \right)dt +K_+\xi^+_T +K_-\xi^-_T \right],
\end{equation}
We assume $K_+$, $K_-$, $\rho$ and $\phi$ strictly positive, $\psi \geq 0$ and $\bar{x}, \, \bar{\theta} \in \R$.
The stationary mean-field control problem is then to determine a policy $\xi^{\star} \in \cB_{mfc}$ so that 
\begin{equation}\label{eq:application:value}
    \bfJ^{mfc}(\xi^{\star}) = \inf_{\xi \in \cB_{mfc} } \bfJ^{mfc}(\xi).
\end{equation}
{\color{black} Note that, for $\xi \in \cB_{mfc}$, the quantity $\theta^\xi$ is well defined and deterministic, since the distribution $p^\xi_\infty$ enters the cost functional only through its first moment. Hence, no finite second-moment condition on $p^\xi_\infty$ is required to define the stationary MFC problem: since both the running cost and the intervention costs are nonnegative, $\bfJ^{mfc}(\xi)$ is well defined as functional taking values in $[0,+\infty]$.
Moreover, the value of the stationary MFC problem is finite, since the null control belongs to $\cB_{mfc}$ and satisfies $\bfJ^{mfc}(0) < \infty$. }

\smallskip
In order to solve the stationary MFC problem, we define an auxiliary stationary mean-field game problem, which we refer to as the potential stationary MFG, whose interaction term is given by two scalar parameters $(\theta,\lambda) \in \R^2$.
{\color{black} We establish both necessary and sufficient conditions relating optimal controls for the stationary MFC problem to solutions of the potential stationary MFG. More precisely, optimal controls for the stationary MFC problem induce solutions of the potential stationary MFG, while the converse holds true within the class of admissible controls satisfying the additional integrability requirement specified below. }

\smallskip
For a control $\xi \in \cB$ and scalars $\theta, \, \lambda \in \R$, we consider the following cost functional, to be minimized:
\begin{equation}\label{eq:mfg:cost}
    \bfJ(\xi,\theta,\lambda) \coloneqq \varlimsup_{T \uparrow \infty}\frac{1}{T}\E\left[ \int_0^T \big( l(X^{\xi}_t,\theta) +\lambda (\theta - X^{\xi}_t) \big)dt +K_+\xi^+_T +K_-\xi^-_T  \right],
\end{equation}
where we have set
\begin{equation}\label{eq:instant_cost}
    l(x,\theta) \coloneqq \frac{1}{2}\rho (x - \bar{x})^2 + \frac{1}{2}\phi (x - \theta)^2 + \frac{1}{2}\psi (\theta - \bar{\theta})^2.
\end{equation}

\begin{definition}\label{def:mfg:solution}
We say that a triple $(\xi^{\star},\theta^\star,\lambda^\star) \in \cB \times \R^2 $ is a solution to the potential stationary MFG, with dynamics \eqref{application:dynamics} and cost functional \eqref{eq:mfg:cost}, if the following holds:
\begin{enumerate}[label=\roman*)]
    \item $\bfJ(\xi^{\star},\theta^\star,\lambda^\star) \leq \bfJ(\xi,\theta^\star,\lambda^\star)$ for any $\xi \in \cB$;
    \item  The optimally controlled state process $X^{\xi^{\star}}$ admits a unique limiting stationary distribution $p^{\star}_\infty \in \cP(\R)$ and
    \begin{align}
        \theta^\star & = \int_{\R} x \, p^{\star}_\infty(dx), \label{mfg:cons:theta}\\
        \lambda^\star & = -\int_{\R} l_\theta(x,\theta^\star) \, p^{\star}_\infty (dx). \label{mfg:cons:lambda}
    \end{align}
\end{enumerate}
\end{definition}

We will need the following assumptions on the parameters:
\begin{assumption}\label{ass:phi_rho}
$\phi < \rho$.
\end{assumption}
Assumption \ref{ass:phi_rho} is needed in the proofs of both the necessary and sufficient condition for the optimality, in order to ensure the existence of equilibria for the stationary MFGs that we will consider in order to solve the stationary MFC problem (see Proposition \ref{prop:mfg:fixed_multiplier} and Theorem \ref{thm:mfg:existence}).
{\color{black} More precisely, this condition is technical in nature: the inequality $\phi<\rho$ guarantees the sign change required in the fixed-point arguments for the equilibrium stationary mean. Thus, the assumption provides a sufficient condition for the existence of the relevant equilibria, rather than identifying a necessary restriction on the interaction strength.}

\smallskip
Our first result, which is proved in Section \ref{sec:necessary}, is a necessary condition for the optimality of a control $\xi^\star \in \cB_{mfc}$ in terms of solutions to the potential stationary MFG:
\begin{theorem}\label{thm:necessary_condition}
Let Assumption \ref{ass:phi_rho} hold true.
Suppose that $\xi^\star \in \cB_{mfc}$ is optimal for the stationary MFC problem \eqref{eq:cost:ergodic}, and let $p^\star_\infty \in \cP(\R)$ denote its stationary distribution.
Set $\theta^\star= \int_{\R} x \, p^\star_\infty(dx)$ and $\lambda^\star = -\int_{\R} l_\theta(x,\theta^\star) \, p^\star_\infty(dx)$.
Then, $(\xi^\star,\theta^\star,\lambda^\star)$ is a solution to the potential stationary MFG.
\end{theorem}

The following Theorem, whose proof is given in Section \ref{sec:sufficient}, ensures that any solution to the potential stationary MFG induces a solution to the stationary MFC problem \eqref{eq:cost:ergodic}, provided that it satisfies {\color{black} extra regularity} conditions and we restrict the integrability requirements on the class of admissible controls:
\begin{theorem}\label{thm:sufficient_condition}
Let $(\xi^{\star},\theta^\star,\lambda^\star)$ be a solution to {\color{black} the} potential stationary MFG, such that
\begin{equation}\label{eq:sufficient_condition:assumption_limit}
    \bfJ(\xi^\star,\theta^\star,\lambda^\star) = \lim_{T \to \infty} \frac{1}{T} \E\left[ \int_0^T \big( l(X^{\xi^\star}_t,\theta^\star) +\lambda^\star (\theta^\star - X^{\xi^\star}_t) \big)dt +K_+\xi^{\star,+}_T +K_-\xi^{\star,-}_T  \right]
\end{equation}
and 
\begin{equation}\label{eq:integrability_control}
    \sup_{T > 0} \frac{1}{T} \E\left[ \int_0^T \vert X^{\xi^\star}_t \vert^2 dt \right] < \infty.
\end{equation}
Then, $\bfJ^{mfc}(\xi^{\star}) \leq \bfJ^{mfc}(\xi)$ for any control $\xi \in \cB_{mfc}$ which satisfies the integrability condition \eqref{eq:integrability_control} as well.
\end{theorem}

Finally, we prove that there exists a solution to the potential stationary MFG.
{\color{black}
By Theorem \ref{thm:sufficient_condition}, this yields the existence of an optimal control for the stationary MFC problem within the class of admissible controls satisfying the additional integrability requirement \eqref{eq:integrability_control}.
}
\begin{theorem}\label{thm:mfg:existence}
Let Assumption \ref{ass:phi_rho} hold true.
There exists a solution to the potential stationary MFG, which satisfies conditions \eqref{eq:sufficient_condition:assumption_limit} and \eqref{eq:integrability_control}.
\end{theorem}
The proof of Theorem \ref{thm:mfg:existence} is given in Section \ref{sec:sufficient}.

\begin{remark}
As the reader may notice, the necessary and the sufficient conditions are obtained in different classes of controls.
Indeed, Theorem \ref{thm:necessary_condition} holds true in the whole class $\cB_{mfc}$, while for Theorem \ref{thm:sufficient_condition} we must restrict to the class of controls in $\cB_{mfc}$ that satisfy the extra integrability condition \eqref{eq:integrability_control}.
This is due to the first-order optimality conditions (see Proposition \ref{prop:first_order}) needed to relate the optimality condition for the potential stationary MFG and the stationary MFC problem, which require extra integrability to hold true.
\end{remark}

\begin{remark}\label{rmk:finite_moments_problem}
The stationary MFC problem can be seen as the stationary version of the following minimization problem:
Find a policy $\xi^\star \in \cB$ which minimizes the cost functional
\begin{equation}\label{eq:cost:finite_moments}
    {\color{black} \widehat{\bfJ}^{mfc}(\xi)} \coloneqq \varlimsup_{T \uparrow \infty}\frac{1}{T}\E\left[\int_0^T \frac{1}{2}\Big(\rho(X^{\xi}_t - \bar{x})^2 + \phi(X^{\xi}_t - \E[X^{\xi}_t])^2 + \psi(\E[X^{\xi}_t] - \bar{\theta})^2 \Big)dt +K_+\xi^+_T +K_-\xi^-_T \right].
\end{equation}
Due to the ergodic cost criterion, we consider the stationary version of the cost functional \eqref{eq:cost:finite_moments} given by \eqref{eq:cost:ergodic}, in which we replace the dependence on the average at time $t$ with the long-time average, i.e. we consider $\theta^{\xi} = \lim_{t \to \infty} \E[X^{\xi}_t]$ instead of $\E[X^{\xi}_t]$ inside the instantaneous cost, for any time $t \geq 0$.
{\color{black}
We will show in Proposition \ref{prop:optimality_time_dependent_mfc} in Section \ref{sec:sufficient} that the solution to the stationary MFC problem constructed in Theorem \ref{thm:mfg:existence} is also a solution to the time-dependent MFC problem within the class of controls in $\cB$ satisfying the additional integrability condition \eqref{eq:integrability_control}.
}
\end{remark}

\section{From mean-field control solutions to mean-field game solutions}\label{sec:necessary}

In this section, we aim at showing that any solution to the stationary MFC problem \eqref{eq:cost:ergodic} is also a solution to the stationary potential MFG of Definition \ref{def:mfg:solution}.
To this end, we introduce the following constrained optimization problem:
for any $\theta \in \R$, define the set
\begin{equation}\label{eq:strategies:constrained}
    \cB_{\theta} \coloneqq \left\{ \xi \in \cB_{mfc}: \, \theta^{\xi} = \theta \right\},
\end{equation}
and set
\begin{equation}\label{eq:value_function:constrained}
    V(\theta) \coloneqq \inf_{\xi \in \cB_{\theta}} \bfJ^{mfc}(\xi).
\end{equation}
We refer to $V(\theta)$ as the value function of the constrained optimization problem.
We notice that $\cB_{mfc} = \cup_{\theta \in \R} \cB_{\theta}$ and 
\[
\inf_{\theta \in \R} V(\theta) =\inf_{\xi \in \cB_{mfc}} \bfJ^{mfc}(\xi).
\]
We connect the constrained optimization problem and the stationary MFC problem as follows: Suppose that $\xi^\star$ is optimal for the stationary MFC problem and set $\theta^\star=\theta^{\xi^\star}$.
Moreover, suppose that the constrained optimization problem \eqref{eq:value_function:constrained} admits an optimum $\xi^\star(\theta)$ for any $\theta$ in $\R$. 
Then, it holds
\[
\bfJ^{mfc}(\xi^\star) =  \inf_{\theta \in \R} \inf_{\xi \in \cB_\theta} \bfJ^{mfc}(\xi) = V(\theta^\star).
\]
Therefore, if $\xi^\star$ is optimal for the stationary MFC problem, its stationary mean $\theta^\star$ {\color{black} minimizes} the value function of the constrained optimization problem.
Provided that the value function $V(\theta)$ is differentiable, a necessary condition for the optimality of $\xi^\star$ is that its stationary mean $\theta^\star$ satisfies $V_\theta(\theta^\star) = 0$.
Thus, the solution strategy to prove the necessary condition is as follows: Solve the constrained optimization problem, show that the value function is differentiable and connect $V_\theta(\theta^\star) = 0$ with the consistency condition \eqref{mfg:cons:lambda} of the potential stationary MFG.

The constrained optimization problem is solved by using a Lagrange multiplier approach.
We introduce the Lagrange multiplier $\lambda \in \R$ and notice that, for any control $\xi \in \cB_\theta$, it holds
\begin{equation}\label{eq:equality_costs}
    \bfJ^{mfc}(\xi) = \bfJ^{mfc}(\xi) + \lambda(\theta - \theta^\xi) = \bfJ(\xi,\theta,\lambda), \quad \forall \, \lambda \in \R,
\end{equation}
with $\bfJ(\xi,\theta,\lambda)$ defined by \eqref{eq:mfg:cost}.
We first show that there exists a control $\xi^\star(\theta,\lambda)$ that solves the unconstrained optimization problem. This is accomplished in Section \ref{sec:solution_ctrl_problem}.
Then, we show that it is possible to choose $\lambda$ so that the optimal control satisfies the constraint. 
Since the solution to the unconstrained problem is optimal in the broader class $\cB$ and, for a suitable choice of $\lambda$, it satisfies the constraint, it is also optimal in the constrained class $\cB_\theta$. This is the content of Section \ref{sec:constrained}.

Having solved the constrained optimization problem, in Section \ref{sec:necessary_final} we finally show that the solution to the stationary MFC problem is also a solution to the potential stationary MFG. 
We prove that the value function $V(\theta)$ is differentiable and we provide a probabilistic representation for its derivative in such a way that it involves the parameters $(\theta,\lambda)$ only.
Then, we show that, if $V_\theta(\theta^\star) = 0$, condition \eqref{mfg:cons:lambda} must hold.
This concludes the proof of the necessary condition.

\begin{remark}
The outlined procedure reveals the nature of the additional parameter $\lambda$ in the potential stationary MFG: $\lambda$ should be the Lagrange multiplier associated to a value $\theta$ that {\color{black} minimizes} the value function $V(\theta)$ of the constrained optimization problem.
Thus,  while the first consistency condition is an admissibility property, stating that the optimal control belongs to $\cB_\theta$ for some $\theta \in \R$, the second consistency condition should be read as an optimality condition on the stationary mean $\theta$, yielding the optimality of the control in $\cB_{mfc}$.
\end{remark}

\begin{remark}
Equation \eqref{eq:equality_costs} can be justified as follows: as $\xi \in \cB_{mfc}$, the state process $X^{\xi}$ is ergodic, and so it holds
\[
\lim_{T \to \infty} \frac{1}{T} \E \left[ \int_0^T \lambda (\theta^{\xi} - X^{\xi}_t) dt \right] = \lambda \left( \theta^{\xi} - \lim_{T \to \infty} \frac{1}{T} \E \left[ \int_0^T X^{\xi}_t dt \right] \right) = \lambda \left( \theta^{\xi} - \int_{\R} x \, p^{\xi}_\infty(dx) \right) = 0.
\]
Then, by exploiting the equality $\varlimsup_{n \to \infty} a_n + \lim_{n \to \infty} b_n = \varlimsup_{n \to \infty}(a_n + b_n)$, for a pair of sequences $(a_n)_{n \geq 1}$, $(b_n)_{n \geq 1}$ with $(b_n)_{n \geq 1}$ convergent, we get \eqref{eq:equality_costs}.
\end{remark}

{\color{black}
\begin{remark}
The strategy adopted in this section does not rely intrinsically on the Ornstein-Uhlenbeck dynamics and on the quadratic cost.
Its key structural ingredients are the ergodic criterion and the scalar form of the stationary interaction, which allow one to use a constrained optimization formulation and a scalar Lagrange multiplier.
In more general settings, one could in principle follow the same program, provided that suitable regularity of the optimal thresholds and of the associated multiplier can be established.
We leave this extension to future research.
\end{remark}
}

\subsection{Solution to the ergodic singular control problem}\label{sec:solution_ctrl_problem}
In this section, we fix $(\theta,\lambda) \in \R^2$ and we solve the ergodic singular control problem with cost functional \eqref{eq:mfg:cost}, by relying on a suitable connection with Dynkin games.
In particular, we prove that, for any $(\theta,\lambda) \in \R^2$, there exists a solution $\xi^{\star}(\theta,\lambda)$ of the ergodic singular control problem with dynamics \eqref{application:dynamics} and cost functional \eqref{eq:mfg:cost}, where $\xi^{\star}(\theta,\lambda)$ is a policy of threshold type.
Moreover, we investigate the regularity of the thresholds with respect to $(\theta,\lambda)$.

\smallskip
For any $(\theta,\lambda) \in \R^2$, we introduce an auxiliary Dynkin game through the upper value function  $U : \R \times \R^2 \to \R$
\begin{equation}\label{eq:dynkin}
U(x,\theta,\lambda) \coloneqq \inf_{\tau \geq 0} \sup_ {\eta \geq 0} \E\left[\int_0^{\tau \land \eta} e^{ -\delta t } \big( l_x(X_t,\theta) - \lambda \big) \, dt + K_- e^{- \delta\tau} \ind_{\tau < \eta} - K_+ e^{-\delta \eta } \ind_{\eta < \tau} \right],
\end{equation}
where $\tau$, $\eta$ are $\bbF$-stopping times and 
\begin{equation}\label{eq:instant_cost_der}
    l_x(x,\theta) = \rho (x - \bar{x}) + \phi (x - \theta) = (\phi +\rho)x - \phi \theta -\rho \bar{x}.
\end{equation}
{\color{black} For $\bbF$-stopping times $\tau$ and $\eta$}, define
\[
M^{\theta,\lambda}_{x}(\tau,\eta) \coloneqq \E\left[\int_0^{\tau \land \eta} e^{ -\delta t } \big( l_x(X_t,\theta) - \lambda \big) \, dt + K_- e^{- \delta \tau } \ind_{\tau < \eta} - K_+ e^{-\delta \eta } \ind_{\eta < \tau} \right],
\]
so that $U(x,\theta,\lambda) = \inf_{\tau \geq 0} \sup_ {\eta \geq 0} M^{\theta,\lambda}_{x}(\tau,\eta)$.
Occasionally, we denote by $X^{x}$ the solution of \eqref{application:dynamics} starting from $x$ at time $0$ to stress the dependence on the initial data.
We observe that $X^{x}$ has the following explicit representation:
\begin{equation}\label{eq:explicit_representation}
    X^{x}_t = e^{-\delta t}\left( x + \sigma\int_0^t e^{\delta s} dW_s \right).
\end{equation}

\smallskip
We start by defining the stopping and continuation regions by  
\begin{equation}\label{eq:stopping_continuation_region}
    \cS^{\theta,\lambda}_+ = \{ x \in \R: \, U(x,\theta,\lambda) \leq -K_+\}, \quad \cS^{\theta,\lambda}_- = \{ x \in \R: \, U(x,\theta,\lambda) \geq K_-\}, \quad \cC^{\theta,\lambda} = \R \setminus (\cS^{\theta,\lambda}_+ \cup \cS^{\theta,\lambda}_-).
\end{equation}
For any $x\in \R$ and $(\theta,\lambda) \in \R^2$, the stopping regions $\cS^{\theta,\lambda}_\pm$ are related to a saddle-point $(\tau^\star,\eta^\star)$ for the Dynkin game \eqref{eq:dynkin}, as ensured by the following Lemma:
\begin{lemma}\label{ctrl:lemma:saddle_point}
It holds 
\[
U(x,\theta,\lambda) =  \sup_ {\eta \geq 0} \inf_{\tau \geq 0} \E\left[\int_0^{\tau \land \eta} e^{ -\delta t } \big( l_x(X^{x}_t,\theta) - \lambda \big) \, dt  + K_- e^{- \delta\tau} \ind_{\tau < \eta} - K_+ e^{-\delta \eta } \ind_{\eta < \tau}  \right]
\]
as well, and that there exists a saddle-point $(\tau^*,\eta^*)$ given by
\begin{equation}\label{eq:dynkin:saddle_point}
    \tau^\star(x,\theta,\lambda) \coloneqq \inf\{ t \geq 0: \, X^{x}_t \in \cS^{\theta,\lambda}_-\}, \quad \eta^\star(x,\theta,\lambda) \coloneqq \inf\{ t \geq 0: \, X^{x}_t \in \cS^{\theta,\lambda}_+\}.
\end{equation}
\end{lemma}
The result follows from a straightforward application of \cite[Theorem 2.1]{peskir2008optimal_stopping_games}, and thus we omit it.

\begin{lemma}\label{lemma:U_monotonicity}
\begin{enumerate}
    \item $x \mapsto U(x,\theta,\lambda)$ is non-decreasing and continuous for any $(\theta,\lambda)$ fixed.
    \item The map $\theta \mapsto U(x,\theta,\lambda)$ is non-increasing for any $(x,\lambda)$ fixed, and the map $\lambda \mapsto U(x,\theta,\lambda)$ is non-increasing for any $(x,\theta)$ fixed.
    \item $U(x,\theta,\lambda)$ is jointly Lipschitz continuous in $(x,\theta,\lambda)$.
\end{enumerate}
\end{lemma}
\begin{proof}
By relying on the explicit representation \eqref{eq:explicit_representation}, it follows that, for $x_1 \leq x_2$, $X^{x_1}_t \leq X^{x_2}_t$ for any $t \geq 0$, $\P$-a.s.
Moreover, since $x \mapsto l_x(x,\theta) - \lambda$ is increasing for any $(\theta,\lambda)$ fixed, $x \mapsto M^{\theta,\lambda}_{x}(\tau,\eta)$ is non-decreasing. This implies that $U(x,\theta,\lambda)$ is non-decreasing as well.
As for $\theta$, we notice that $l_{x \theta}(x,\theta) = - \phi < 0$, so that one has that $\theta \mapsto M^{\theta,\lambda}_{x}(\tau,\eta)$ is non-increasing for any $(\tau,\eta)$, which implies that $\theta \mapsto U(x,\theta,\lambda)$ is non-increasing as well. The monotonicity with respect to $\lambda$ is proven analogously.

As for continuity, let $(x_n,\theta_n,\lambda_n) \to (x,\theta,\lambda)$. Fix $\epsilon > 0$ and let $\bar\eta$ so that
\begin{equation}\label{filtered:lemma_regularity_U:leq_bound}
U(x,\theta,\lambda) = \sup_{\eta}\inf_{\tau} M^{\theta,\lambda}_{x}(\tau,\eta) \leq \inf_{\tau} M^{\theta,\lambda}_{x}(\tau,\bar\eta) + \epsilon \leq M^{\theta,\lambda}_{x}(\tau,\bar\eta) + \epsilon \quad \forall \tau,
\end{equation}
where the first equality holds by Lemma \ref{ctrl:lemma:saddle_point}.
Analogously, let $\tau^n$ so that 
\begin{equation}\label{filtered:lemma_regularity_U:geq_bound}
U(x_n,\theta_n,\lambda_n) =  \inf_{\tau} \sup_{\eta} M^{\theta_n,\lambda_n}_{x_n}(\tau,\eta) \geq \sup_{\eta} M^{\theta_n,\lambda_n}_{x_n}(\tau^n,\eta) - \epsilon \geq M^{\theta_n,\lambda_n}_{x_n}(\tau^n,\eta) - \epsilon \quad \forall \eta.
\end{equation}
By taking the differences, we get
\begin{align*}
    U(x,\theta,\lambda) & - U(x_n,\theta_n,\lambda_n) \leq M^{\theta,\lambda}_{x}(\tau^n,\bar{\eta}) - M^{\theta^n,\lambda^n}_{x_n}(\tau^n,\bar{\eta}) +2\epsilon \\
    & = \E\left[\int_0^{\tau^n \land \bar{\eta}} e^{-\delta t } \left( l_x(X^{x}_t,\theta) - l_x(X^{x_n}_t,\theta^n) + ( \lambda^n - \lambda ) \right) dt \right] +2\epsilon \\
    & \leq \E\left[\int_0^\infty e^{-\delta t }  \left( (\phi + \rho) \vert X^{x_n}_t - X^{x}_t \vert + \phi \vert \theta^n - \theta \vert + \vert \lambda^n - \lambda \vert \right) dt \right] +2\epsilon.
\end{align*}
By the same reasoning, exchanging the roles of $U(x,\theta,\lambda)$ and $U(x_n,\theta_n,\lambda_n)$ in \eqref{filtered:lemma_regularity_U:leq_bound} and \eqref{filtered:lemma_regularity_U:geq_bound}, we bound the absolute value of the difference by
\begin{align*}
    \vert U(x,\theta,\lambda) & - U(x_n,\theta_n,\lambda_n) \vert \leq C\E\left[\int_0^\infty e^{-\delta t } \left( \vert X^{x}_t - X^{x_n}_t\vert + \vert \theta^n - \theta \vert + \vert \lambda^n - \lambda \vert \right)  dt \right] +2\epsilon \\
    & \leq C \left(  \vert x - x_n \vert \int_0^\infty e^{-2\delta t } dt + \left(\vert \theta^n - \theta \vert + \vert \lambda^n - \lambda \vert\right) \int_0^\infty e^{-\delta t } dt  \right) +2\epsilon \overset{n \to \infty}{\longrightarrow} 2\epsilon
\end{align*}
with $C > 0$ a positive constant, where we used the explicit representation \eqref{eq:explicit_representation} to handle the dependence on $x$.
By arbitrariness of $\epsilon$, we get the desired continuity.
\end{proof}

We notice that Lemma \ref{lemma:U_monotonicity} implies that the stopping regions $\cS_\pm^{\theta,\lambda}$ are closed and the continuation region $\cC^{\theta,\lambda}$ is open for any $(\theta,\lambda)$.
For any $(\theta,\lambda) \in \R^2$ define the functions
\begin{equation}\label{eq:barriers}
    a_-(\theta,\lambda) \coloneqq \inf\{ x \in \R: \; U(x,\theta,\lambda) \geq K_- \}, \quad a_+(\theta,\lambda) \coloneqq \sup\{ x \in \R: \; U(x,\theta,\lambda) \leq -K_+ \}.
\end{equation}
By continuity and monotonicity of $U(\cdot,\theta,\lambda)$, exploiting the bounds $-K_+ \leq U(x,\theta,\lambda) \leq K_-$, the sets $\cS^{\theta,\lambda}_\pm$ and $\cC^{\theta,\lambda}$ in \eqref{eq:stopping_continuation_region} can be expressed in terms of $a_\pm(\theta,\lambda)$ as
\begin{equation}\label{eq:filtered:regions_U}
\begin{aligned}
    & \cS^{\theta,\lambda}_+ = \{  x \in \R: \, x \leq a_+(\theta,\lambda) \}, \quad \cS^{\theta,\lambda}_- = \{  x \in \R: \, x \geq a_-(\theta,\lambda) \}, \\
    & \cC^{\theta,\lambda} = \{  x \in \R: \, a_+(\theta,\lambda) < x < a_-(\theta,\lambda) \}.
\end{aligned}
\end{equation}

We notice that there exist {\color{black} two unique functions} $x_\pm(\theta,\lambda)$, $x_+(\theta,\lambda) < x_-(\theta,\lambda)$ for any $(\theta,\lambda)$, so that
\begin{equation}\label{eq:unique_zeros}
    l_x(x,\theta) - \lambda  - K_- \delta \; \left\{ \begin{aligned}
        & > 0, \quad \forall x > x_+(\theta,\lambda),  \\
        & = 0, \quad \forall x = x_+(\theta,\lambda), \\
        & < 0, \quad \forall x < x_+(\theta,\lambda),
    \end{aligned} \right. \qquad     l_x(x,\theta) - \lambda  + K_+ \delta \; \left\{ \begin{aligned}
        & > 0, \quad \forall x > x_-(\theta,\lambda),  \\
        & = 0, \quad \forall x = x_-(\theta,\lambda), \\
        & < 0, \quad \forall x < x_-(\theta,\lambda).
    \end{aligned} \right.
\end{equation}
In particular, the functions $x_\pm(\theta,\lambda)$ are given by
\begin{equation}\label{eq:implicit_functions}
\left\{ \begin{aligned}
    & x_-(\theta,\lambda) = \frac{1}{(\phi + \rho)}\left( \phi\theta + \lambda  +\rho\bar{x} +K_-\delta \right),  \\
    & x_+(\theta,\lambda) = \frac{1}{(\phi + \rho)}\left( \phi\theta + \lambda  +\rho\bar{x} -K_+\delta \right).
\end{aligned} \right.
\end{equation}
This allows us to prove that the barriers $a_\pm(\theta,\lambda)$ are well separated, in the following sense.
\begin{lemma}\label{lemma:theta:boundary_separeted}
It holds $-\infty < a_+(\theta,\lambda) \leq x_+(\theta,\lambda) < x_-(\theta,\lambda) \leq a_-(\theta,\lambda) < \infty$.
\end{lemma}
\begin{proof}
By \cite[Theorem 2.1]{peskir2008optimal_stopping_games}, the processes
\begin{align}
    & \Big( e^{-\delta (t \land \eta^*) }U(X_{t \land \eta^*}, \theta,\lambda) + \int_0^{t\land \eta^*}e^{- \delta s } \big( l_x(X_s,\theta) -\lambda \big) ds \Big)_{t \geq 0},  \label{eq:lemma_regularity:subm} \\
    & \Big( e^{-\delta (t \land \tau^*) }U(X_{t \land \tau^*},\theta,\lambda) + \int_0^{t\land \tau^*}e^{- \delta s} \big(l_x (X_s,\theta) - \lambda \big) ds \Big)_{t \geq 0} \label{eq:lemma_regularity:supm}
\end{align}
are, respectively, a sub-martingale and a super-martingale under $\P_{x}$.
Let now $x \in \cS^{\theta,\lambda}_-$.
By using the sub-martingale property of the process \eqref{eq:lemma_regularity:subm} and the bound $U(x,\theta,\lambda) \leq K_-$, we deduce
\begin{align*}
& K_- = U(x,\theta,\lambda) \leq \E\left[ e^{-\delta (t \land \eta^*) }U(X_{t \land \eta^*},\theta,\lambda) + \int_0^{t\land \eta^*}e^{- \delta s } \big( l_x(X_s,\theta) -\lambda \big) ds \right] \\
& \leq \E\left[ K_- e^{-\delta (t \land \eta^*) } + \int_0^{t\land \eta^*}e^{- \delta s } \big( l_x(X_s,\theta) -\lambda \big) ds  \right] = K_- + \E\left[  \int_0^{t\land \eta^*}e^{- \delta s } \big( l_x(X_s,\theta) -\lambda -K_- \delta \big) ds  \right]
\end{align*}
for any $t \geq 0$.
This implies
\begin{equation*}
    0 \leq \lim_{t \to 0}\frac{1}{t}\E\left[ \int_0^{t}e^{- \delta s } \big( l_x(X_s,\theta) -\lambda -K_- \delta \big) \ind_{s < \eta^*} ds \right] = l_x(x,\theta) - \lambda -K_- \delta , \;\; \forall x \in \cS^{\theta,\lambda}_-,
\end{equation*}
where we used dominated convergence theorem, by virtue of the representation \eqref{eq:explicit_representation}.
Repeating the same argument for $x \in \cS^{\theta,\lambda}_+$, by relying on the super-martingale property of the process \eqref{eq:lemma_regularity:supm}, we deduce
\begin{equation}\label{eq:barriers:necessarily_separated}
    \left\{ \begin{aligned}
        & l_x(x,\theta) - \lambda -K_-\delta \geq 0, \quad \forall x \in \cS^{\theta,\lambda}_-,  \\
        & l_x(x,\theta) - \lambda +K_+\delta \leq 0, \quad \forall x \in \cS^{\theta,\lambda}_+.
    \end{aligned} \right.
\end{equation}
Equation \eqref{eq:barriers:necessarily_separated} yields
\begin{equation}
\left\{ \begin{aligned}
    & a_-(\theta,\lambda) \geq \inf\{ x \in \R: l_x(x,\theta) -\lambda -K_- \delta  \geq 0 \} = x_-(\theta,\lambda), \\
    & a_+(\theta,\lambda) \leq \sup\{ x \in \R: l_x(x,\theta) -\lambda +K_+ \delta  \leq 0 \} = x_+(\theta,\lambda),
\end{aligned} \right.
\end{equation}
with $x_+(\theta,\lambda) < x_-(\theta,\lambda)$, which implies that $a_+(\theta,\lambda) < a_-(\theta,\lambda)$ as well.

\smallskip
To show that $a_\pm(\theta,\lambda)$ are never infinity, we show by contradiction that $\cS^{\theta,\lambda}_\pm$ are non-empty.
Indeed, suppose that $\cS^{\theta,\lambda}_+ = \emptyset$.
Then, for any $ x \in \R$, it holds $-K_+ < U(x,\theta,\lambda)$, which implies that $\eta^\star(x,\theta,\lambda) = \infty$ $\P$-a.s.
Therefore, 
\begin{multline*}
-K_+ < U(x,\theta,\lambda) = \inf_{\tau}\E\left[\int_0^{\tau} e^{-\delta t} \big( l_x(X_t,\theta) - \lambda \big) dt + K_- e^{-\delta \tau}\ind_{\tau < \infty} \right] \\
\leq \E\left[\int_0^\infty e^{-\delta t } \big( l_x(X_t,\theta) -\lambda \big) dt \right] \leq C_1 x + C^{\theta,\lambda},
\end{multline*}
for all $x \in \R$, where $C^{\theta,\lambda}$ is a positive constant depending on $(\theta,\lambda)$ and $C_1$ is a positive constant independent of $(\theta,\lambda)$.
This yields
\[
-K_+ < \lim_{x \to -\infty}  \E\left[\int_0^\infty e^{- \delta t } \big( l_x(X_t,\theta) - \lambda \big) dt \right] \leq \lim_{x \to -\infty} \big( C_1 x + C^{\theta,\lambda} \big) = -\infty,
\]
thus getting a contradiction.
The proof of $a_-(\theta,\lambda) < +\infty$ is dealt with analogously.
\end{proof}
We now investigate the regularity of the thresholds. In particular, we show that they are Lipschitz with respect to the parameters $(\theta,\lambda) \in \R^2$.
\begin{lemma}\label{lemma:theta:boundary_regularity}
\begin{enumerate}
    \item The maps $\theta \mapsto a_\pm(\theta,\lambda)$ {\color{black} are} non-decreasing for any $\lambda$ fixed, and $\lambda \mapsto a_\pm(\theta,\lambda)$ {\color{black} are} non-decreasing for any $\theta$ fixed.
    \item Set
    \begin{equation}\label{eq:constant_barriers}
    m=\frac{\phi}{\rho + \phi}, \qquad c = \frac{1}{\rho + \phi}.
    \end{equation}
    The maps $\theta \mapsto a_\pm(\theta,\lambda)$ are bi-Lipschitz, in the sense that
    \begin{equation}\label{eq:bilipschitz:theta}
        m \vert \theta - \theta' \vert \leq \vert a_\pm(\theta,\lambda) - a_\pm(\theta',\lambda) \vert \leq 2m \vert \theta - \theta' \vert
    \end{equation}
    for any $\theta,\theta' \in \R$, $\lambda \in \R$.
    Analogously, the maps $\lambda \mapsto a_\pm(\theta,\lambda)$ are bi-Lipschitz, in the sense that
    \begin{equation}\label{eq:bilipschitz:lambda}
        c \vert \lambda - \lambda' \vert \leq \vert a_\pm(\theta,\lambda) - a_\pm(\theta,\lambda') \vert \leq 2c \vert \lambda - \lambda' \vert
    \end{equation}
    for any $\lambda, \lambda' \in \R$, $\theta \in \R$.
    \item The maps $(\theta,\lambda) \mapsto a_\pm(\theta,\lambda)$ are jointly Lipschitz, in the sense that \begin{equation}\label{eq:barriers:lipschitz}
        \vert a_\pm(\theta,\lambda) - a_\pm(\theta',\lambda') \vert \leq 2m \vert \theta - \theta' \vert + 2c \vert \lambda - \lambda' \vert
    \end{equation}
    for any $(\theta,\lambda), (\theta',\lambda') \in \R^2$.
    \item $a_\pm(\theta,\lambda)$ is differentiable for a.e. $(\theta,\lambda) \in \R^2$, and it holds $m \leq \partial_\theta a_\pm(\theta,\lambda) \leq 2m $ and $c \leq \partial_\lambda a_\pm(\theta,\lambda) \leq 2c $.
\end{enumerate}
\end{lemma}
\begin{proof}
We start with proving the monotonicity of $a_\pm$. Let $\theta_1 > \theta_2$. 
Since $\theta \mapsto U(x,\theta,\lambda)$ is non-increasing, we have
\[
\{ x \in \R: U(x,\theta_2,\lambda) \leq -K_+ \} \subseteq \{ x \in \R: U(x,\theta_1,\lambda) \leq -K_+ \}.
\]
By taking the supremum of these sets and relying on the definition \eqref{eq:barriers} of $a_+(\theta,\lambda)$, one gets $a_+(\theta_1,\lambda) \geq a_+(\theta_2,\lambda)$.
The monotonicity of the other maps can be proved analogously.

Fix $\lambda \in \R$ and $\theta_1 < \theta_2$. We show that it holds
\begin{equation}\label{eq:double_estimate_theta}
    U\big( x + m(\theta_2 - \theta_1),\theta_2,\lambda \big) \leq U\big(x,\theta_1,\lambda\big) \leq U\big(x + 2m(\theta_2 - \theta_1),\theta_2,\lambda \big).
\end{equation}
Indeed, let $(\tau,\eta)$ be two $\bbF$-stopping times. By \eqref{eq:explicit_representation}, we have $X^{x+r}_t - X^{x}_t = e^{-\delta t} r$, for any $x \in \R$, $t \geq 0$. Thus, we have
\begin{multline}\label{eq:estimate_M_theta}
    M^{\theta_2,\lambda}_{x+r}(\tau,\eta) - M^{\theta_1,\lambda}_{x}(\tau,\eta) = \E\left[ \int_0^{\tau\land\eta} e^{-\delta t}\left( l_x(X^{x+r}_t,\theta_2) - l_x(X^{x}_t,\theta_1) \right)dt \right] \\
    = \E\left[ \int_0^{\tau\land\eta} e^{-\delta t} \Big( (\phi + \rho)e^{-\delta t} r - \phi(\theta_2-\theta_1) \Big) dt \right].
\end{multline}
Now, choose $r = m(\theta_2-\theta_1)$, so that \eqref{eq:estimate_M_theta} implies
\[
M^{\theta_2,\lambda}_{x+m(\theta_2 - \theta_1) }(\tau,\eta) - M^{\theta_1,\lambda}_{x}(\tau,\eta) = \phi (\theta_2 - \theta_1) \E\left[ \int_0^{\tau\land\eta} e^{-\delta t}( e^{-\delta t} -1 ) dt \right] \leq 0,
\]
since $\delta > 0$, $\phi > 0$ and $\theta_2 > \theta_1$.
Thus, we have $M^{\theta_2,\lambda}_{x+m(\theta_2-\theta_1)}(\tau,\eta) \leq  M^{\theta_1,\lambda}_{x}(\tau,\eta)$ for any $(\tau,\eta)$. Taking the supremum in $\eta \geq 0$ and the infimum in  $\tau \geq 0$ implies the first estimate in \eqref{eq:double_estimate_theta}.
Consider now $r = 2m(\theta_2-\theta_1)$.
By \eqref{eq:estimate_M_theta}, we have
\[
M^{\theta_2,\lambda}_{x+2m(\theta_2 - \theta_1) }(\tau,\eta) - M^{\theta_1,\lambda}_{x}(\tau,\eta) = \phi (\theta_2 - \theta_1) \E\left[ \int_0^{\tau\land\eta} (2e^{-2\delta t} -  e^{-\delta t}) dt \right].
\]
Since, for any $u \geq 0$, it holds
\begin{equation}\label{eq:inequality_delta}
\int_0^u e^{-\delta t}dt = \frac{1-e^{-\delta u}}{\delta} \leq 2 \frac{ 1-e^{-2\delta u} }{2\delta} = 2\int_0^u e^{-2\delta t}dt,
\end{equation}
we deduce that $M^{\theta_2,\lambda}_{x+2m(\theta_2 - \theta_1) }(\tau,\eta) - M^{\theta_1,\lambda}_{x}(\tau,\eta) \geq 0$ for any $(\tau,\eta)$, which implies, by taking the supremum in $\eta \geq 0$ and the infimum in $\tau \geq 0$, that the second estimate in \eqref{eq:double_estimate_theta} holds.
Finally, \eqref{eq:double_estimate_theta} implies the bi-Lipschitz property of $\theta \mapsto a_\pm(\theta,\lambda)$. We show it for $a_+(\theta,\lambda)$, as it is completely analogous for $a_-(\theta,\lambda)$.
By definition of $a_+(\theta_1,\lambda)$ and the first estimate in \eqref{eq:double_estimate_theta}, we have $-K_+ = U(a_+(\theta_1,\lambda),\theta_1,\lambda) \geq U( a_+(\theta_1,\lambda)+m(\theta_2-\theta_1),\theta_2,\lambda )$, which implies that $a_+(\theta_1,\lambda)+m(\theta_2-\theta_1) \in \cS^{\theta_2,\lambda}_+$. Therefore, it holds $a_+(\theta_1,\lambda)+m(\theta_2-\theta_1) \leq a_+(\theta_2,\lambda)$, i.e. $m(\theta_2-\theta_1) \leq a_+(\theta_2,\lambda) - a_+(\theta_1,\lambda)$.
On the other hand, take $x > a_+(\theta_1,\lambda)$. Then, $-K_+ < U(x,\theta_1,\lambda) \leq U(x+2m(\theta_2-\theta_1),\theta_2,\lambda)$, which implies, by definition of $a_+(\theta,\lambda)$, $x+2m(\theta_2-\theta_1) > a_+(\theta_2,\lambda)$. Taking the limit as $x \downarrow a_+(\theta_1,\lambda)$, we deduce $a_+(\theta_1,\lambda) + 2m(\theta_2-\theta_1) \geq a_+(\theta_2,\lambda)$, i.e. the second estimate in \eqref{eq:bilipschitz:theta}.

As for the dependence on $\lambda$, fix now $\theta \in \R$ and let $\lambda_1 < \lambda_2$. Then, it holds
\begin{equation}\label{eq:double_estimate:lambda}
    U(x + c(\lambda_2 - \lambda_1),\theta,\lambda_2) \leq U(x,\theta,\lambda_1) \leq U(x+2c(\lambda_2-\lambda_1),\theta,\lambda_2).
\end{equation}
Indeed, by the same computations as above, it is possible to show that, for any $r \in \R$, it holds
\begin{equation}\label{eq:estimate_M_lambda}
    M^{\theta,\lambda_2}_{x+r}(\tau,\eta) - M^{\theta,\lambda_1}_{x}(\tau,\eta) = \E\left[ \int_0^{\tau\land\eta} e^{-\delta t} \Big( (\phi + \rho)e^{-\delta t} r - (\lambda_2-\lambda_1) \Big) dt \right].
\end{equation}
Then, taking $r = c(\lambda_2 - \lambda_1)$, it follows $M^{\theta,\lambda_2}_{x+c(\lambda_2-\lambda_1)}(\tau,\eta) - M^{\theta,\lambda_1}_{x}(\tau,\eta) \leq 0$, while taking $r = 2c(\lambda_2-\lambda_1)$ and exploiting again the inequality \eqref{eq:inequality_delta}, it follows $M^{\theta,\lambda_2}_{x+2c(\lambda_2-\lambda_1)}(\tau,\eta) - M^{\theta,\lambda_1}_{x}(\tau,\eta) \geq 0$.
Finally, by the same reasoning as for $\theta \mapsto a_\pm(\theta,\lambda)$, the bi-Lipschitz property \eqref{eq:bilipschitz:lambda} of $\lambda \mapsto a_\pm(\theta,\lambda)$ follows.
By putting together \eqref{eq:bilipschitz:theta} and \eqref{eq:bilipschitz:lambda}, we get to \eqref{eq:barriers:lipschitz}.

Finally, since the maps $(\theta,\lambda) \mapsto a_\pm(\theta,\lambda)$ are Lipschitz, they are a.e. differentiable by Rademacher's theorem. The bounds on the partial derivatives follow from the monotonicity of $\theta \mapsto a_\pm(\theta,\lambda)$ and $\lambda \mapsto a_\pm(\theta,\lambda)$ together with the bi-Lipschitz properties \eqref{eq:bilipschitz:theta} and \eqref{eq:bilipschitz:lambda}.
\end{proof}

We recall that the infinitesimal generator $\cL_{X^0}$ of the diffusion $X^0$ is given by
\[
\cL_{X^0}f(x) = -\delta x f_{x}(x) + \frac{1}{2}\sigma^2 f_{xx}(x),
\]
for any $f \in \dC^2_b(\R)$.
The following Lemma shows that the Dynkin game value $U(x,\theta,\lambda)$ solves a free-boundary problem.
\begin{lemma}\label{lemma:Dynkin_game:solution}
For any $(\theta,\lambda) \in \R^2$, one has $U(\cdot,\theta,\lambda) \in \dC^1(\R) \cap \dC^2(\R \setminus \{ a_+(\theta,\lambda),a_-(\theta,\lambda) \})$.
Moreover, $U(\cdot,\theta,\lambda)$ solves the equation
\begin{equation}\label{eq:free_boundary_problem}
    \left\{ \begin{aligned}
        -\delta U(x,\theta,\lambda) + \cL_{X^0}U(x,\theta,\lambda) + l_x(x,\theta) - \lambda = 0, \quad & a_+(\theta,\lambda) < x < a_-(\theta,\lambda), \\
        U(x,\theta,\lambda) = -K_+, \quad & x \leq a_+(\theta,\lambda), \\
        U(x,\theta,\lambda) = K_-, \quad & x \geq a_-(\theta,\lambda).
    \end{aligned} \right.
\end{equation}
Finally, it holds
\begin{align}
        -\delta U(x,\theta,\lambda) + \cL_{X^0}U(x,\theta,\lambda) + l_x(x,\theta) - \lambda \geq 0, \quad &  x \geq a_+(\theta,\lambda),  \label{eq:U_signs:positive} \\
        -\delta U(x,\theta,\lambda) + \cL_{X^0}U(x,\theta,\lambda) + l_x(x,\theta) - \lambda \leq 0, \quad &  x \leq a_-(\theta,\lambda). \label{eq:U_signs:negative}
\end{align}
\end{lemma}
\begin{proof}
$U(\cdot,\theta,\lambda)$ belongs to $\dC^\infty$ in the interior of $\cS^{\theta,\lambda}_\pm$ as it is constant therein.
As for the behavior in $\cC^{\theta,\lambda}$, we recall that the flow $x \mapsto X^{x}_t$ is a diffeomorphism for every $t > 0$, by \cite[Theorem 13.8]{rogerswilliams_vol2}.
Then, for any $f \in \dC^\infty_c(\R)$ and $t > 0$, the map $x \mapsto \E[f(X^{x}_t)]$ is $\dC^2_b(\R)$.
Therefore, Assumption (3.5) in \cite{peskir2025weak} is satisfied, which ensures by Corollary 5 therein that $U(\cdot,\theta,\lambda)$ satisfies
\begin{equation}\label{eq:filtered:unconstrained_PDE}
- \delta U(x,\theta,\lambda) + \cL_{X^0}U(x,\theta,\lambda)+ l_x(x,\theta) -\lambda = 0
\end{equation}
inside $\cC^{\theta,\lambda}$, where the derivatives appearing in \eqref{eq:filtered:unconstrained_PDE} are to be understood in the sense of Schwartz distribution.
Then, it is enough to notice that the generator $\cL_{X^0}$ is uniformly elliptic on $\R$ and $l_x(\cdot,\theta,\lambda) \in \dC^{\infty}(\R)$, which implies that $U(\cdot,\theta,\lambda) \in \dC^{2}(\cC^{\theta,\lambda})$ and that $U(\cdot,\theta,\lambda)$ satisfies \eqref{eq:filtered:unconstrained_PDE} in the classical sense in $\cC^{\theta,\lambda}$.
To conclude the $\dC^1$ regularity, we verify that $U(\cdot,\theta,\lambda)$ is $\dC^1$ at $a_\pm(\theta,\lambda)$.
Since $U(\cdot,\theta,\lambda)$ is $\dC^2$ over $\cC^{\theta,\lambda}$ and $U_x(x,\theta,\lambda) \equiv 0$ in the interior of $\cS^{\theta,\lambda}_\pm$, we check that $U_x(x_n,\theta,\lambda) \to 0$ for any $(x_n)_{n \geq 1} \subseteq \cC^{\theta,\lambda}$, $x_n \to a_\pm(\theta,\lambda)$.
To this extent, we first show that, for any $(\theta,\lambda) \in \R^2$ and $x \in \cC^{\theta,\lambda}$, it holds
\begin{equation}\label{eq:value_dynkin:derivatives}
    U_x(x,\theta,\lambda) = (\rho+\phi)\E\left[\int_0^{\tau^\star \land \eta^\star}e^{-2 \delta t} dt \right],
\end{equation}
with $(\tau^\star,\eta^\star)$ given by \eqref{eq:dynkin:saddle_point}.
Indeed, as $x \mapsto U(x,\theta,\lambda)$ is non-decreasing and continuous for any $(\theta,\lambda) \in \R^2$ fixed, there exists $\epsilon_0 > 0$ so that $x+\epsilon \in \cC^{\theta,\lambda}$ for any $0 \leq \epsilon \leq \epsilon_0$.
Denote by $(\tau^\star_{x},\eta^\star_{x})$ the saddle-point for the Dynkin game with initial value $x$, and by $(\tau^\star_{x+\epsilon},\eta^\star_{x+\epsilon})$ the saddle-point for the Dynkin game with initial value $x+\epsilon$.
As the Dynkin game has a value, it holds
\begin{equation}\label{eq:lemma_regularity:U_bounds}
\begin{aligned}
    M^{\theta,\lambda}_{x}(\tau^\star_{x},\eta^\star_{x+\epsilon}) \leq U(x,\theta,\lambda) \leq M^{\theta,\lambda}_{x}(\tau^\star_{x+\epsilon},\eta^\star_{x}), \\
    M^{\theta,\lambda}_{x+\epsilon}(\tau^\star_{x+\epsilon},\eta^\star_{x}) \leq U(x+\epsilon,\theta,\lambda) \leq M^{\theta,\lambda}_{x+\epsilon}(\tau^\star_{x},\eta^\star_{x+\epsilon}).
\end{aligned}
\end{equation}
By \eqref{eq:lemma_regularity:U_bounds}, we deduce
\begin{multline}\label{eq:U_epsilon_bound_above_x}
    \frac{1}{\epsilon}\left( U(x+\epsilon,\theta,\lambda) - U(x,\theta,\lambda) \right) \leq \frac{1}{\epsilon}\E\left[ \int_0^{\tau^\star_{x} \land \eta^\star_{x+\epsilon}} e^{-\delta t } \big(l_x(X^{x+\epsilon}_t,\theta) - l_x(X^{x}_t, \theta) \big) dt \right] \\
    = (\rho+\phi)\E\left[ \int_0^{\tau^\star_{x} \land \eta^\star_{x+\epsilon}} e^{-2\delta t} dt \right]
\end{multline}
and
\begin{multline}\label{eq:U_epsilon_bound_below_x}
    \frac{1}{\epsilon}\left( U(x+\epsilon,\theta,\lambda) - U(x,\theta,\lambda) \right) \geq \frac{1}{\epsilon}\E\left[ \int_0^{\tau^\star_{x+\epsilon} \land \eta^\star_{x}} e^{-\delta t} \big(l_x(X^{x+\epsilon}_t,\theta) - l_x(X^{x}_t,\theta) \big) dt \right] \\
    = (\rho+\phi) \E\left[ \int_0^{\tau^\star_{x+\epsilon} \land \eta^\star_{x}} e^{-2\delta t} dt \right]
\end{multline}
As $(\tau^\star_{x+\epsilon},\eta^\star_{x+\epsilon}) \to (\tau^\star_{x},\eta^\star_{x})$ as $\epsilon \to 0$, we conclude that $U_x$ is given by \eqref{eq:value_dynkin:derivatives}.
Then, it is enough to notice that
\[
\tau^\star(x_n,\theta,\lambda) \to 0 \quad \text{and} \quad \eta^\star(x_n,\theta,\lambda) \to 0
\]
for any $x_n \to a_+(\theta,\lambda)$ or $x_n \to a_-(\theta,\lambda)$, respectively.
Thus, $U(\cdot,\theta,\lambda) \in \dC^1(\R) \cap \dC^2(\R \setminus \{ a_+(\theta,\lambda),a_-(\theta,\lambda) \})$, and it satisfies \eqref{eq:free_boundary_problem}.
Finally, to see \eqref{eq:U_signs:positive}, suppose that $x \geq a_+(\theta,\lambda)$.
If $x < a_-(\theta,\lambda)$, then \eqref{eq:U_signs:positive} is actually an equality, since $U$ satisfies \eqref{eq:free_boundary_problem}.
If $x > a_-(\theta,\lambda)$, then $\cL_{X^0} U(x,\theta,\lambda) = 0$ and $U(x,\theta,\lambda) = K_-$, so that \eqref{eq:U_signs:positive} becomes $-\delta K_- + l_x(x,\theta) -\lambda$, which is non-negative by \eqref{eq:barriers:necessarily_separated}.
This proves \eqref{eq:U_signs:positive}.
\eqref{eq:U_signs:negative} is proven analogously.
\end{proof}

\begin{proposition}\label{prop:ergodic_ctrl:solution}
For any  $(\theta,\lambda) \in \R^2$, let $(X^{\xi^{\star}(\theta,\lambda)},\xi^{\star}(\theta,\lambda))$ be the solution of the Skorohod reflection problem 
\begin{equation}\label{eq:skorohod_reflection}
\left\{ \begin{aligned}
    & a_+(\theta,\lambda) \leq X^{\xi^{\star}(\theta,\lambda)}_t \leq a_-(\theta,\lambda), \quad \P_{x}\text{-a.s.,} \;\; \text{for almost all }t \geq 0,\\
    & \int_0^t \ind_{\{ X^{\xi^{\star}(\theta,\lambda)}_{s} > a_+(\theta,\lambda) \}}d \xi^{\star,+}_s(\theta,\lambda) = \int_0^t \ind_{\{ X^{\xi^{\star}(\theta,\lambda)}_{s} < a_-(\theta,\lambda) \}}d \xi^{\star,-}_s(\theta,\lambda) = 0, \quad \P_{x}\text{-a.s.,} \;\; \forall t \geq 0.
\end{aligned} \right.
\end{equation}
Then, $\xi^{\star}(\theta,\lambda) \in \cB$ is optimal, and it holds
\begin{multline}\label{eq:final_ctrl}
    \inf_{\xi \in \cB} \bfJ(\xi,\theta,\lambda) = \lim_{T \to \infty} \frac{1}{T}\E\left[ \int_0^T \big( l(X^{\xi^{\star}(\theta,\lambda)}_t,\theta) + \lambda (\theta - X^{\xi^{\star}(\theta,\lambda)}_t ) \big)dt + K_+ \xi^{\star,+}(\theta,\lambda) + K_- \xi^{\star,-}(\theta,\lambda)\right] = \\
    K_+\delta a_+(\theta,\lambda) +  l\big( a_+(\theta,\lambda),\theta) + \lambda ( \theta - a_+(\theta,\lambda)).
\end{multline}
\end{proposition}
\begin{proof}
We connect the Dynkin game \eqref{eq:dynkin} with the optimization problem of $\bfJ({\color{black} \xi},\theta,\lambda)$ \eqref{eq:mfg:cost}.
To this end, we define the functions
\begin{equation}\label{eq:ctrl_problem:potential_value}
        v(x,\theta,\lambda) \coloneqq \int_{a_+(\theta,\lambda)}^x U(x',\theta,\lambda)dx', \qquad \kappa(\theta,\lambda) \coloneqq K_+\delta a_+(\theta,\lambda) +  l\big( a_+(\theta,\lambda),\theta)  + \lambda ( \theta - a_+(\theta,\lambda)).
\end{equation}
We have that $v(\cdot,\theta,\lambda) \in \dC^2(\R)$ and has at most linear growth, since $U(\cdot,\theta,\lambda) \in \dC^1(\R)$ and $-K_+ \leq U(x,\theta,\lambda) \leq K_-$.
Relying on \eqref{eq:free_boundary_problem}, \eqref{eq:U_signs:positive} and \eqref{eq:U_signs:negative}, a direct computation shows that the pair $(v(\cdot,\theta,\lambda)$, $\kappa(\theta,\lambda))$ solves the Hamilton-Jacobi-Bellman equation for the ergodic singular control problem, i.e.
\begin{equation}\label{eq:HJB}
\min\{\cL_{X^0} v(x,\theta,\lambda) + l(x,\theta) +\lambda( \theta - x) - \kappa(\theta,\lambda), -v_x(x,\theta,\lambda) + K_-, v_x(x,\theta,\lambda) + K_+\}=0, \quad  \forall \, x \in \R.
\end{equation}
Moreover, we notice that there exists a unique solution $(X^{\xi^{\star}(\theta,\lambda)},\xi^{\star}(\theta,\lambda))$ to the Skorohod reflection problem \eqref{eq:skorohod_reflection} (see, e.g., \cite[Theorem 4.1]{tanaka1979reflecting}).
Then, a standard verification theorem (see, e.g., \cite[Theorem 3.2]{jack2006singular}) yields
\[
\inf_{\xi \in \cB}\bfJ(\xi,\theta,\lambda) = \bfJ(\xi^\star(\theta,\lambda))=\kappa(\theta,\lambda),
\]
i.e. $\xi^\star(\theta,\lambda)$ is optimal and $\kappa(\theta,\lambda)$ is the value of the ergodic problem.
Finally, we notice that, by \cite[Lemma 2.1]{alvarez2018stationary}, the limsup in the cost criterion is actually a limit.
This concludes the proof.
\end{proof}

\subsection{Solution to the constrained optimization problem}\label{sec:constrained}

In this section, we solve the constrained optimization problem \eqref{eq:value_function:constrained}.

\begin{definition}\label{def:lambda_mfg}
Let $\lambda \in \R$. We say that a pair $(\xi^\star(\lambda),\theta^\star(\lambda))$ in $\cB \times \R$ is a solution of the $\lambda$-stationary MFG if the following {\color{black} two properties} hold:
\begin{enumerate}
    \item $\bfJ(\xi^\star(\lambda),\theta^\star(\lambda),\lambda) \leq \bfJ(\xi,\theta^\star(\lambda),\lambda)$ for any $\xi \in \cB$, where $\bfJ(\xi,\theta,\lambda)$ is given by \eqref{eq:mfg:cost};
    \item The optimally controlled state process $X^{\xi^{\star}(\lambda)}$ admits a unique stationary distribution $p^{\lambda}_\infty \in \cP(\R)$ and
    \begin{equation}\label{lambda_mfg:cons}
        \theta^\star(\lambda) = \int_{\R} x \, p^{\lambda}_\infty(dx).
    \end{equation}
\end{enumerate}
\end{definition}
We show in Proposition \ref{prop:mfg:fixed_multiplier} that the $\lambda$-stationary MFG admits a solution for any $\lambda$ fixed.
Then, Proposition \ref{prop:mfg:fixed_mean} shows that, for any $\theta \in \R$ it is possible to choose $\lambda(\theta)$ so that the optimal control belongs to $\cB_\theta$.
Finally, in Proposition \ref{prop:solution_constrained_problem}, we conclude that the constrained optimization problem admits a solution, and we provide a useful representation of the value function $V(\theta)$ in terms of the unconstrained optimization problem, which proves crucial in concluding the proof of Theorem \ref{thm:necessary_condition}.

\begin{proposition}\label{prop:mfg:fixed_multiplier}
Fix $\lambda \in \R$. 
Then, there exists a solution $(\xi^\star(\lambda),\theta^\star(\lambda))$ of the $\lambda$-stationary MFG.
\end{proposition}
\begin{proof}
By Proposition \ref{prop:ergodic_ctrl:solution}, for any $(\theta,\lambda) \in \R^2$ fixed, there exists a control $\xi^{\star}(\theta,\lambda)$ solution to the ergodic singular control problem \eqref{eq:mfg:cost}.
By \cite[Paragraph 36]{borodin2002handbook}, the diffusion $X^{\xi^{\star}(\theta,\lambda)}$ admits a unique stationary distribution $p^{\theta,\lambda}_\infty \in \cP(\R)$, given by
\begin{equation}\label{eq:stationary_distribution}
    p^{\theta,\lambda}_\infty(dx) = \frac{1}{\int_{a_+(\theta,\lambda)}^{a_-(\theta,\lambda)} p'(y) dy } 1_{[a_+(\theta,\lambda), a_-(\theta,\lambda)]}(x) p'(x)dx,
\end{equation}
where $p'(x)$ is the density of the speed measure associated to the Ornstein-Uhlenbeck process $X^0$,  given by
\begin{equation}\label{OU:speed_measure}
    p'(x) = \frac{2}{\sigma^2}e^{-\frac{\delta}{\sigma^2}x^2}.
\end{equation}
Define the map ${\color{black} G^\lambda}:\R \to \R$ by setting
\begin{equation}\label{eq:fixed_lambda:fixed_point_map}
    {\color{black} G^\lambda}(\theta) = \int_{\R} x \, p^{\theta,\lambda}_\infty(dx) =   \frac{1}{\int_{a_+(\theta,\lambda)}^{a_-(\theta,\lambda)} p'(y) dy } \int_{a_+(\theta,\lambda)}^{a_-(\theta,\lambda)}x p'(x)dx.
\end{equation}
We search for a {\color{black} fixed point} of $G^\lambda(\theta)$.
Exploiting $\P ( a_+(\theta,\lambda) \leq X^{\xi^{\star}(\theta,\lambda)}_t \leq a_-(\theta,\lambda) ) = 1$ for any $t \geq 0$, we get
\begin{equation}\label{eq:fixed_point:bounds}
    a_+(\theta,\lambda) \leq \varlimsup_{T \uparrow \infty} \frac{1}{T} \int_0^T \E[X^{\xi^{\star}(\theta,\lambda)}_t] dt = G^\lambda(\theta) \leq a_-(\theta,\lambda) .
\end{equation}
Since the map $\theta \mapsto G^\lambda(\theta) - \theta$ is continuous for any fixed $\lambda$, it is enough to show that there exist $\theta_-$ such that $G^\lambda(\theta_-) - \theta_- < 0$ and $\theta_+$ such that $G^\lambda(\theta_+) - \theta_+ > 0$, with $\theta_+ < \theta_-$.
By continuity, this implies the existence of $\theta^\star(\lambda) \in [\theta_+,\theta_-]$ such that $G^\lambda(\theta^\star(\lambda)) = \theta^\star(\lambda)$.
To do so, we exploit the bounds in \eqref{eq:fixed_point:bounds} and we show that there exist $\theta_-$ such that $a_-(\theta_-,\lambda) - \theta_- < 0$ and $\theta_+$ such that $a_+(\theta_+,\lambda) - \theta_+ > 0$.
Take $\theta < 0$. By Lemma \ref{lemma:theta:boundary_regularity}, since $\theta \mapsto a_+(\theta,\lambda)$ is non-decreasing and bi-Lipschitz, the map $\theta \mapsto a_+(\theta,\lambda)$ is a.e. differentiable and it holds $m \leq \partial_\theta a_+(\theta,\lambda) \leq 2m$.
Thus, for any $\theta < 0 $, we have
\begin{align*}
    a_+(\theta,\lambda) =  a_+(0,\lambda) + \int_0^\theta \partial_\theta a_+(\theta',\lambda) d\theta' = a_+(0,\lambda) - \int_\theta^0 \partial_\theta a_+(\theta',\lambda) d\theta' \geq a_+(0,\lambda) + 2m\theta.
\end{align*}
This implies
\[
a_+(\theta,\lambda) - \theta \geq (2m-1)\theta + a_+(0,\lambda) > 0.
\]
Since $\phi < \rho$ by Assumption \ref{ass:phi_rho}, we have $2m < 1$. Thus, the previous inequality is satisfied by $\theta_+ < \frac{a_+(0,\lambda)}{1 - 2m} \wedge 0$.
Analogously for $a_-(\theta,\lambda)$, take $\theta > 0$ and consider the upper bound
\begin{align*}
    a_-(\theta,\lambda) = a_-(0,\lambda) + \int_\theta^0 \partial_\theta a_-(\theta',\lambda) d\theta' \leq a_-(0,\lambda) + 2m\theta,
\end{align*}
so that $a_-(\theta,\lambda) - \theta < 0$ if $(1 - 2m)\theta > a_-(0,\lambda)$.
Since $2m < 1$, the previous inequality is satisfied by $\theta_- > \frac{a_-(0,\lambda)}{1 - 2m} \vee 0$.
\end{proof}

\begin{proposition}\label{prop:mfg:fixed_mean}
For any $\theta \in \R$, there exists a unique $\lambda(\theta)$ and a solution $(\xi^\star(\lambda),\theta^\star(\lambda))$ to the $\lambda$-stationary MFG such that $\theta^\star(\lambda(\theta)) = \theta$.
\end{proposition}
\begin{proof}
Fix $\theta \in \R$ and consider the map $\lambda \mapsto F^{\theta}(\lambda) \coloneqq G^\lambda(\theta)$, with $G^\lambda(\theta)$ defined by \eqref{eq:fixed_lambda:fixed_point_map}.
{\color{black} Since the maps $\lambda \mapsto a_\pm(\theta,\lambda)$ are Lipschitz continuous, the map $\lambda \mapsto F^\theta(\lambda)$ is locally Lipschitz, hence absolutely continuous on compact intervals and differentiable a.e.}
Recalling that $\partial_\lambda a_\pm(\theta,\lambda) \geq c$ a.e., a direct computation yields
\begin{multline*}
    F^{\theta}_\lambda(\lambda) =  \left(  \int_{a_+(\theta,\lambda)}^{a_-(\theta,\lambda)} p'(x) dx \right)^{-2} \Bigg( \partial_\lambda a_+(\theta,\lambda) p'(a_+(\theta,\lambda)) \int_{a_+(\theta,\lambda)}^{a_-(\theta,\lambda)} ( x - a_+(\theta,\lambda) )  p'(x) dx \\
    + \partial_\lambda a_-(\theta,\lambda) p'(a_-(\theta,\lambda)) \int_{a_+(\theta,\lambda)}^{a_-(\theta,\lambda)} ( a_-(\theta,\lambda) - x )  p'(x) dx \Bigg) > 0
\end{multline*}
This yields that $F^{\theta}(\lambda)$ is strictly increasing, and so injective.
By definition of $F^{\theta}(\lambda)$ and \eqref{eq:fixed_point:bounds}, we have
\[
a_+(\theta,\lambda) \leq F^{\theta}(\lambda) \leq a_-(\theta,\lambda).
\]
For any $\theta \in \R$ fixed, we have $\lim_{\lambda \to \pm \infty} a_\pm(\theta,\lambda) = \pm \infty$, which implies that $\lim_{\lambda \to \pm \infty} F^{\theta}(\lambda) = \pm \infty$ as well.
Indeed, for $\lambda > 0$ we have
\[
a_+(\theta,\lambda) = a_+(\theta,0) + \int_0^\lambda \partial_\lambda a_+(\theta,\lambda')d\lambda' \geq a_+(\theta,0) + c\lambda
\]
which goes to $+\infty$ as $\lambda \to \infty$.
Analogously, take $\lambda < 0$ and bound $a_-(\theta,\lambda)$ by
\[
a_-(\theta,\lambda) = a_-(\theta,0) + \int_0^\lambda \partial_\lambda a_-(\theta,\lambda')d\lambda' = a_-(\theta,0) - \int_\lambda^0 \partial_\lambda a_-(\theta,\lambda')d\lambda' \leq {\color{black} a_-(\theta,0)} + c\lambda,
\]
which implies that, for fixed $\theta \in \R$, it holds $\lim_{\lambda \to -\infty} a_-(\theta,\lambda) = -\infty$.
Since $\lambda \mapsto F^{\theta}(\lambda)$ is continuous, strictly increasing, and $\lim_{\lambda\to\pm\infty}F^{\theta}(\lambda)=\pm\infty$, it is a bijection from $\R$ onto $\R$.
Hence there exists a unique $\lambda(\theta)$ such that $F^{\theta}(\lambda(\theta))=\theta$.
This concludes the proof.
\end{proof}

\begin{proposition}\label{prop:solution_constrained_problem}
For any $\theta \in \R$, there exists a solution to the constrained optimization problem \eqref{eq:value_function:constrained}.
Moreover, it holds
\begin{equation}\label{eq:value_function:unconstrained}
    V(\theta) = \inf_{\xi \in \cB_\theta}\bfJ^{mfc}(\xi) = \inf_{ \xi \in \cB } \bfJ\big(\xi,\theta,\lambda(\theta)\big).
\end{equation}
\end{proposition}
\begin{proof}
For any $\theta \in \R$, by Propositions \ref{prop:mfg:fixed_multiplier} and \ref{prop:mfg:fixed_mean}, there exist a control $\xi^\star(\theta,\lambda(\theta)) \in \cB$ and a unique $\lambda(\theta) \in \R$ such that
\[
G^{\lambda(\theta)}(\theta) = \theta, \quad \bfJ(\xi^\star(\theta,\lambda(\theta)),\theta,\lambda(\theta) ) \leq \bfJ(\xi,\theta,\lambda(\theta) ), \; \forall \xi \in \cB.
\]
This implies that $\xi^\star(\theta,\lambda(\theta)) \in \cB_\theta$ and, in virtue of \eqref{eq:equality_costs}, $\bfJ^{mfc}(\xi^\star(\theta,\lambda(\theta))) \leq \bfJ^{mfc}(\xi)$ for any $\xi \in \cB_\theta$.
Thus, $\xi^\star(\theta,\lambda(\theta))$ is optimal for the constrained optimization problem.
To see equality \eqref{eq:value_function:unconstrained}, it is enough to notice that the control $\xi^\star(\theta,\lambda(\theta))$ is optimal in the class $\cB$ for the control problem $\bfJ(\cdot,\theta,\lambda(\theta))$ by Proposition \ref{prop:ergodic_ctrl:solution} and by choice of $\lambda(\theta)$.
\end{proof}

\subsection{Proving the necessary condition}\label{sec:necessary_final}

We conclude the proof of Theorem \ref{thm:necessary_condition} by deriving the probabilistic representation of the derivative of the value function of the constrained optimization $V(\theta)$.
To this extent, we first prove that $\lambda(\theta)$ is locally Lipschitz (see Lemma \ref{lemma:lambda_differentiable}), hence a.e. differentiable.
Then, we show that $V(\theta)$ is everywhere differentiable
(see Lemma \ref{lemma:value_function:differentiability}).
Then, we relate the derivative of $V$ and of the cost functional $\bfJ(\xi,\theta,\lambda(\theta))$, relying on the envelope theorem of \cite{milgrom2002envelope}.
Finally, we show that $V_\theta(\theta^\star) = 0$ if and only if \eqref{mfg:cons:lambda} holds, thus concluding the proof.

\begin{lemma}\label{lemma:lambda_differentiable}
$\lambda(\theta)$ is locally Lipschitz, hence differentiable a.e.
\end{lemma}
\begin{proof}
Define the map $Q:\R^2 \to \R$ by setting
\begin{equation}
    Q(\theta,\lambda) = \int_{a_+(\theta,\lambda)}^{a_-(\theta,\lambda)}(x - \theta)p'(x) dx.
\end{equation}
We notice that $Q(\theta,\lambda(\theta)) = 0$ is equivalent to $\theta=G^{\lambda(\theta)}(\theta)$, for $G^\lambda$ defined by \eqref{eq:fixed_lambda:fixed_point_map}.
We notice also that $Q$ is a locally Lipschitz map, being $a_\pm(\theta,\lambda)$ Lipschitz and $(\alpha,\beta,\theta) \mapsto \int_\alpha^\beta (x-\theta)p'(x)dx$ locally Lipschitz.
By Proposition \ref{prop:mfg:fixed_mean}, for any $\theta_0 \in \R$ there exists a unique $\lambda_0$ so that $Q(\theta_0,\lambda_0) = 0$.
We want to show that there exists an open neighborhood $I$ of $(\theta_0,\lambda_0)$ where it holds
\begin{equation}\label{eq:bound_Q_lambda}
    Q_\lambda(\theta,\lambda) \geq \kappa, \quad \text{for a.e $(\theta,\lambda) \in I$}.
\end{equation}
This implies that we can apply the generalized implicit function theorem for 
locally Lipschitz functions (see \cite[Section 7.1, Corollary at p. 256]{clarke1990optimization}) to deduce that there exists a neighborhood $U$ of $\theta_0$ and a function $\Tilde{\lambda}(\theta)$ Lipschitz on $U$ so that $Q(\theta,\Tilde{\lambda}(\theta)) = 0$.
Since for any $\theta$ there exists a unique $\lambda(\theta)$ so that $Q(\theta,\lambda(\theta))=0$ by Proposition \ref{prop:mfg:fixed_mean}, we have $\Tilde{\lambda}(\theta) = \lambda(\theta)$, which implies that $\lambda(\theta)$ is locally Lipschitz, hence differentiable a.e.

To prove \eqref{eq:bound_Q_lambda}, we adopt the following reasoning: Let $U=\{(\theta,\lambda) \in \R^2: \, a_+(\theta,\lambda) < \theta < a_-(\theta,\lambda) \}$ and notice that, by {\color{black} joint continuity} of $a_\pm(\theta,\lambda)$, it is an open set.
Since $a_\pm(\theta,\lambda)$ are a.e. differentiable, $Q(\theta,\lambda)$ is itself a.e. differentiable. In particular, for a.e. $(\theta,\lambda) \in U$ it holds
\begin{equation}\label{eq:sign_q_lambda2}
\begin{aligned}
    Q_\lambda(\theta,\lambda) & = \partial_\lambda a_-(\theta,\lambda) \big( a_-(\theta,\lambda) - \theta \big) p'\big( a_-(\theta,\lambda) \big) +   \partial_\lambda a_+(\theta,\lambda) \big(\theta - a_+(\theta,\lambda) \big) p'\big( a_+(\theta,\lambda) \big) \\
    & \geq c \Big( \big( a_-(\theta,\lambda) - \theta \big) p'\big( a_-(\theta,\lambda) \big) + \big(\theta - a_+(\theta,\lambda) \big) p'\big( a_+(\theta,\lambda) \big) \Big) =: H(\theta,\lambda),
\end{aligned}
\end{equation}
where {\color{black} the estimate holds} since, for $(\theta,\lambda) \in U$, we have $a_-(\theta,\lambda) - \theta > 0$ and $\theta - a_+(\theta,\lambda) > 0$.
Observe that $H(\theta,\lambda)$ is jointly continuous in $(\theta,\lambda)$ since $a_\pm(\theta,\lambda)$ are jointly Lipschitz continuous by \eqref{eq:barriers:lipschitz} and $p'(x)$ defined by \eqref{OU:speed_measure} is a continuous function.
Suppose now that $(\theta_0,\lambda_0)$ is so that $Q(\theta_0,\lambda_0)=0$.
Since $Q(\theta_0,\lambda) = 0$ is equivalent to $G^\lambda(\theta) = \theta$, it holds $a_+(\theta_0,\lambda_0) < \theta_0 < a_-(\theta_0,\lambda_0)$, i.e. $(\theta_0,\lambda_0) \in U$.
This implies that
\[
H(\theta_0,\lambda_0) = c \Big( \big( a_-(\theta_0,\lambda_0) - \theta_0 \big) p'\big( a_-(\theta_0,\lambda_0) \big) + \big(\theta_0 - a_+(\theta_0,\lambda_0) \big) p'\big( a_+(\theta_0,\lambda_0) \big) \Big) > 0.
\]
Since $H(\theta,\lambda)$ is continuous, there exists a neighborhood $V$ of $(\theta_0,\lambda_0)$ such that $H(\theta,\lambda) > \frac{1}{2} H(\theta_0,\lambda_0)$.
We now take $I = U \cap V$ and notice that $(\theta_0,\lambda_0) \in I$.
Then, for a.e. $(\theta,\lambda) \in I$, it holds
\[
Q_\lambda(\theta,\lambda) \geq H(\theta,\lambda) \geq \frac{1}{2}H(\theta_0,\lambda_0) > 0.
\]
This concludes the proof.
\end{proof}

We need the following technical result, which ensures that, for a control $\xi \in \cB_{mfc}$ satisfying {\color{black} extra integrability} conditions, the cost functional $\bfJ$ is differentiable with respect to $\theta$ and $\lambda$, with explicit partial derivatives.
\begin{lemma}\label{lemma:diffentiability_cost}
Let $\xi \in \cB_{mfc}$ be such that the integrability condition \eqref{eq:integrability_control} is satisfied {\color{black} and $\bfJ(\xi,\theta,\lambda) < \infty$ for every $(\theta,\lambda)\in\R^2$}.
Then, {\color{black} the map} $(\theta,\lambda) \mapsto \bfJ(\xi,\theta,\lambda)$ is differentiable, and its partial derivatives are given by
\begin{equation}\label{eq:derivative_J:theta}
\begin{aligned}
    & \partial_\theta \bfJ(\xi,\theta,\lambda) \coloneqq \lim_{h \to 0}\frac{1}{h}\left( \bfJ(\xi,\theta+h,\lambda) - \bfJ(\xi,\theta,\lambda) \right) = \lim_{T \to \infty} \frac{1}{T} \E\left[\int_0^T (l_\theta(X^{\xi}_t,\theta) + \lambda )dt \right], \\
    & \partial_\lambda \bfJ(\xi,\theta,\lambda) = \lim_{h \to 0}\frac{1}{h}\left( \bfJ(\xi,\theta,\lambda+h) - \bfJ(\xi,\theta,\lambda) \right) = \lim_{T \to \infty} \frac{1}{T} \E\left[\int_0^T (\theta - X^\xi_t)dt \right].
\end{aligned}
\end{equation}
\end{lemma}
\begin{proof}
The proof follows from an application of the Moore-Osgood theorem {\color{black}(see, e.g., \cite[Chapter 4, Section 11, Theorem 2]{zakon2011mathematical})}.
Let $\xi$ be as in the statement of the lemma.
Let $h \in \R$, $h \neq 0$.
{\color{black} For any $T > 0$, set
\[
    \bfJ_T(\xi,\theta,\lambda) \coloneqq \frac{1}{T}\E\left[\int_0^T \Big( l(X^{\xi}_t,\theta) + \lambda(\theta - X^{\xi}_t) \Big)dt + K_+\xi^+_T + K_-\xi^-_T \right],
\]
so that $\bfJ(\xi,\theta,\lambda) = \varlimsup_{T \to \infty}\bfJ_T(\xi,\theta,\lambda)$.
Set
\[
    f(h,T) \coloneqq \frac{1}{T}\E\left[\int_0^T \Big( \frac{1}{h} \big( l(X^{\xi}_t,\theta+h) - l(X^{\xi}_t,\theta) \big) + \lambda \Big)dt \right], \qquad g(T) \coloneqq \frac{1}{T}\E\left[\int_0^T \Big( l_\theta(X^{\xi}_t,\theta) + \lambda \Big)dt \right].
\]
For every $T > 0$, it holds
\[
    \bfJ_T(\xi,\theta+h,\lambda) = \bfJ_T(\xi,\theta,\lambda) + h f(h,T).
\]
Moreover, since $l_\theta(x,\theta) = \phi(\theta-x)+\psi(\theta-\bar{\theta})$ is affine in $x$, the ergodicity of $X^\xi$ implies that $\lim_{T \to \infty} g(T)$ exists. Since
\[
    f(h,T) = g(T) + \frac{\phi+\psi}{2}h,
\]
the limit $\lim_{T \to \infty} f(h,T)$ exists as well, for every fixed $h \neq 0$. Therefore, using the equality $\varlimsup_{T \to \infty}(a_T+b_T) = \varlimsup_{T \to \infty}a_T + \lim_{T \to \infty}b_T$ whenever $(b_T)_{T>0}$ converges, we get
\[
    \frac{1}{h}\left( \bfJ(\xi,\theta+h,\lambda) - \bfJ(\xi,\theta,\lambda) \right) = \frac{1}{h}\left( \varlimsup_{T \to \infty}\big( \bfJ_T(\xi,\theta,\lambda) + h f(h,T) \big) - \varlimsup_{T \to \infty}\bfJ_T(\xi,\theta,\lambda) \right) = \lim_{T \to \infty}f(h,T).
\] }
We show that $\lim_{h \to 0}f(h,T) = g(T)$ uniformly in $T > 0$.
Then, by the Moore-Osgood theorem, it holds
\[
    \lim_{h \to 0} \frac{1}{h}\left(\bfJ(\xi,\theta+h,\lambda) - \bfJ(\xi,\theta,\lambda) \right) = \lim_{h \to 0} \lim_{T \to \infty} f(h,T) = \lim_{T \to \infty} g(T) = \lim_{T \to \infty}\frac{1}{T}\E\left[\int_0^T \Big( l_\theta(X^{\xi}_t,\theta) + \lambda \Big)dt \right],
\]
where the last limit exists since $X^\xi$ is ergodic as $\xi \in \cB_{mfc}$ and $l_\theta(x,\theta) = \phi(\theta - x) + \psi(\theta - \bar{\theta})$ is affine in $x$.
To see the uniform convergence, we notice that it holds
\[
    f(h,T) - g(T) = \frac{1}{T} \E\left[\int_0^T \Big( \int_0^1 l_\theta(X^{\xi}_t,\theta + s h)ds - l_\theta(X^{\xi}_t,\theta) \Big)dt \right] = \frac{\phi + \psi}{2}h
\]
for any $T > 0$.
As this obviously goes to $0$ uniformly in $T$ as $h \to 0$, the first identity in \eqref{eq:derivative_J:theta} follows.
{\color{black} The second identity in \eqref{eq:derivative_J:theta} follows analogously.
Finally, using the ergodicity of $X^\xi$, we can write
\[
    \partial_\theta \bfJ(\xi,\theta,\lambda) = \phi(\theta-\theta^\xi)+\psi(\theta-\bar\theta)+\lambda, \qquad \partial_\lambda \bfJ(\xi,\theta,\lambda) = \theta-\theta^\xi.
\]
Hence the partial derivatives are continuous in $(\theta,\lambda)$, and therefore the map $(\theta,\lambda) \mapsto \bfJ(\xi,\theta,\lambda)$ is differentiable.}
\end{proof}

\begin{lemma}\label{lemma:value_function:differentiability}
At any point $\theta \in \R$, the function $V(\theta)$ is differentiable and it holds
\begin{equation}\label{eq:derivative_V}
    V_\theta(\theta) = \int_{\R} l_\theta(x,\theta) p^{\theta,\lambda(\theta)}_\infty(dx) + \lambda(\theta).
\end{equation}
\end{lemma}
\begin{proof}
Let $\theta \in \R$ and let $\xi^\star(\theta) \coloneqq \xi^\star(\theta,\lambda(\theta))$ as given by Proposition \ref{prop:ergodic_ctrl:solution}.
Notice that $\xi^\star(\theta) \in \cB_{mfc}$, with stationary distribution given by \eqref{eq:stationary_distribution} and it satisfies condition \eqref{eq:integrability_control} since it is always bounded by the constants $a_\pm(\theta,\lambda(\theta))$.
Thus, by Lemma \ref{lemma:diffentiability_cost}, for any $(\theta,\lambda) \in \R^2$ the partial derivatives with respect to $\theta$ and $\lambda$ of $\bfJ$ at $(\xi^\star(\theta),\theta,\lambda)$ exist, and they are given by \eqref{eq:derivative_J:theta}.
Moreover, by Proposition \ref{prop:ergodic_ctrl:solution} and \eqref{eq:ctrl_problem:potential_value}, we have
\[
V(\theta)  = \kappa(\theta,\lambda(\theta)) = K_+\delta a_+(\theta,\lambda(\theta)) +  l\big( a_+(\theta,\lambda(\theta)),\theta) + \lambda(\theta)(\theta - a_+(\theta,\lambda(\theta))).
\]
Since $a_+(\theta,\lambda)$ is Lipschitz in $(\theta,\lambda)$ and $\lambda(\theta)$ is locally Lipschitz, we conclude that $V$ is differentiable a.e.
Let now $\theta \in \R$ be a point of differentiability of both $V(\theta)$ and $\lambda(\theta)$.
By applying the envelope theorem (see, e.g., \cite[Theorem 1]{milgrom2002envelope}), it holds
\begin{align*}
    V_\theta(\theta) & = \partial_\theta \bfJ(\xi,\theta,\lambda(\theta))\Big\vert_{\xi = \xi^{\star}(\theta)} = \partial_\theta \bfJ(\xi^{\star}(\theta),\theta,\lambda(\theta)) + \partial_\lambda \bfJ(\xi^{\star}(\theta),\theta,\lambda(\theta)) \lambda'(\theta) \\
    & = \lim_{T \to \infty} \frac{1}{T}\E\left[ \int_0^T \big( l_\theta(X^{\xi^{\star}(\theta)}_t,\theta) +\lambda(\theta) \big) dt \right] + \lambda'(\theta) \lim_{T \to \infty} \frac{1}{T}\E\left[\int_0^T (\theta - X^{\xi^\star(\theta)}_t) dt \right] \\
    & = \lim_{T \to \infty} \frac{1}{T}\E\left[ \int_0^T \big( l_\theta(X^{\xi^{\star}(\theta)}_t,\theta) +\lambda(\theta) \big) dt \right],
\end{align*}
which can be rewritten as \eqref{eq:derivative_V}, where in the last equality we exploited the fact that $\xi^\star(\theta)$ has stationary mean given by $\theta$ itself.
As the right-hand side of \eqref{eq:derivative_V} is continuous and coincides with $V_\theta(\theta)$ for a.e. $\theta$, we deduce that $V$ is differentiable everywhere with derivative given by \eqref{eq:derivative_V}.
\end{proof}

We can now conclude the proof of Theorem \ref{thm:necessary_condition}:

\begin{proof}[Proof of Theorem \ref{thm:necessary_condition}]
If $\theta^\star$ is a minimum point of $V$, it holds $V_\theta(\theta^\star) = 0$, which, by \eqref{eq:derivative_V}, implies \eqref{mfg:cons:lambda}.
\end{proof}

\section{From mean-field game solutions to mean-field control solutions}\label{sec:sufficient}
{\color{black}
In this section we prove Theorem \ref{thm:sufficient_condition}, showing that any solution to the potential stationary MFG satisfying the stated conditions yields an optimal control for the stationary MFC problem within the class of admissible controls satisfying the additional integrability requirement \eqref{eq:integrability_control}.}
Moreover, we prove that the potential stationary MFG admits an equilibrium so that conditions \eqref{eq:sufficient_condition:assumption_limit} and \eqref{eq:integrability_control} are satisfied, thus providing an explicit solution to the stationary MFC problem {\color{black} within this restricted class.}

\smallskip
We start this section by proving the following first-order optimality conditions for the ergodic singular control problem, inspired by \cite[Lemma 4.3]{cannerozzi2024cooperation}.
Notice that these first order optimality conditions are tailor-made for our dynamics and, most importantly, for the ergodic cost criterion.

\begin{proposition}\label{prop:first_order}
Let $(\theta,\lambda) \in \R^2$ and let $\xi^{\star}(\theta,\lambda)$ be optimal for $\bfJ(\cdot,\theta,\lambda)$.
Suppose that $\xi^\star(\theta,\lambda)$ satisfies the integrability condition \eqref{eq:integrability_control}. 
Then, for any other $\xi \in \cB$ that satisfies  \eqref{eq:integrability_control} as well, it holds
\begin{equation}\label{eq:optimality:first_order_condition}
    \varliminf_{T \to \infty} \frac{1}{T}\E\left[ \int_0^T \big( l_x(X^{\xi^{\star}}_t,\theta) - \lambda \big) (X^{\xi^{\star}}_t - X^{\xi}_t)  + K_+ (\xi^{\star,+}_T(\theta,\lambda) - \xi^{+}_T) + K_- (\xi^{\star,-}_T(\theta,\lambda) - \xi^{-}_T)   \right] \leq 0.
\end{equation}
\end{proposition}
\begin{proof}
To simplify the notation, we omit the dependence on $(\theta,\lambda)$ in the optimal control $\xi^\star$, and we set $h(x) \coloneqq l(x,\theta) +\lambda(\theta- x)$.
For any $\xi \in \cB$, $\epsilon \in (0,\frac{1}{2}]$, set $\xi^\epsilon = \epsilon \xi+ (1-\epsilon)\xi^{\star} = \xi^{\star} + \epsilon (\xi - \xi^{\star})$.
We define
\begin{multline*}
    f(\epsilon,T) = \frac{1}{\epsilon}\bigg(  \frac{1}{T}\E\left[ \int_0^T h(X^{\xi^\epsilon}_t) dt + K_+ \xi^{\epsilon,+}_T + K_- \xi^{\epsilon,-}_T \right]
    -  \frac{1}{T}\E\left[ \int_0^T h(X^{\xi^{\star}}_t)dt + K_+ \xi^{\star,+}_T + K_- \xi^{\star,-}_T\right]  \bigg).
\end{multline*}
We notice that $X^{\xi}_t - X^{\xi^\star}_t = \frac{1}{\epsilon}(X^{\xi^\epsilon}_t - X^{\xi^\star}_t)$ for any $\epsilon \in (0,\frac{1}{2}]$.
Then, for any $\epsilon \in (0,\frac{1}{2}]$ and $T > 0$, it holds
\begin{multline*}
    f(\epsilon,T) \\
    = \frac{1}{T} \left( \E\left[ \int_0^T \Big( \int_0^1 h_{x}(X^{\xi^\star}_t + \epsilon\tau (X^{\xi}_t - X^{\xi^\star}_t))d\tau  \Big) (X^{\xi}_t - X^{\xi^\star}_t) dt +K_+(\xi^{+}_T - \xi^{\star,+}_T) +K_-(\xi^{-}_T - \xi^{\star,-}_T)  \right] \right),
\end{multline*}
which implies
\begin{equation}\label{first_order:dominated} 
    \lim_{\epsilon \downarrow 0} f(\epsilon,T) = \frac{1}{T} \E\left[ \int_0^T h_{x}(X^{\xi^\star}_t) (X^{\xi}_t - X^{\xi^\star}_t) dt +K_+(\xi^{+}_T - \xi^{\star,+}_T) +K_-(\xi^{-}_T - \xi^{\star,-}_T) \right]
\end{equation}
uniformly in $T$.
Indeed, by exploiting the linearity of $h_x(x)$ in $x$, we have
\begin{multline*}
    f(\epsilon,T ) - \frac{1}{T} \E\left[ \int_0^T h_{x}(X^{\xi^\star}_t) (X^{\xi}_t - X^{\xi^\star}_t) dt +K_+(\xi^{+}_T - \xi^{\star,+}_T) +K_-(\xi^{-}_T - \xi^{\star,-}_T) \right]  \\
    = \frac{1}{T} \E\left[ \int_0^T (\phi + \rho) (X^{\xi}_t - X^{\xi^\star}_t) \int_0^1 \big( \epsilon\tau (X^{\xi}_t - X^{\xi^\star}_t) \big) d\tau dt  \right] = \frac{1}{2T} \E\left[ \int_0^T  (\phi + \rho) (X^{\xi}_t - X^{\xi^\star}_t)^2 dt  \right] \epsilon.
\end{multline*}
Then, it follows
\begin{multline}\label{lemma:first_order_condition:uniform_estimate}
    \lim_{\epsilon \downarrow 0} \sup_{T \geq 0} \left\vert f(\epsilon,T ) - \frac{1}{T} \E\left[ \int_0^T h_{x}(X^{\xi^\star}_t) (X^{\xi}_t - X^{\xi^\star}_t) dt +K_+(\xi^{+}_T - \xi^{\star,+}_T) +K_-(\xi^{-}_T - \xi^{\star,-}_T) \right] \right\vert \\
    \leq  (\phi + \rho) \sup_{T \geq 0} \frac{1}{T} \E\left[ \int_0^T  \vert X^{\xi^\star}_t \vert^2 +\vert X^{\xi}_t \vert^2 dt  \right] \lim_{\epsilon \downarrow 0} \epsilon  = 0
\end{multline}
where we exploited the integrability requirement \eqref{eq:integrability_control}.
On the other hand, by taking the limit with respect to $T$, it holds
\begin{align*}
& \varlimsup_{T \uparrow \infty} f(\epsilon,T) \geq \frac{1}{\epsilon}\left(  \varlimsup_{T \uparrow \infty}\frac{1}{T}\E\left[ \int_0^T h(X^{\xi^\epsilon}_t) dt + K_+ \xi^{\epsilon,+}_T + K_- \xi^{\epsilon,-}_T \right] \right.\\
& \left. - \varlimsup_{T \uparrow \infty}\frac{1}{T}\E\left[ \int_0^T h(X^{\xi^\star}_t)dt + K_+ \xi^{\star,+}_T + K_- \xi^{\star,-}_T \right] \right) = \frac{1}{\epsilon} ( \bfJ(\xi^\epsilon,\theta,\lambda) - \bfJ(\xi^{\star},\theta,\lambda)) \geq 0,    
\end{align*}
by using the inequality $\varlimsup_n (a_n - b_n) \geq \varlimsup_n a_n - \varlimsup_n b_n $ and by optimality, for any $\epsilon \in (0,\frac{1}{2}]$.
By employing \cite[Lemma 7.5]{cannerozzi2024cooperation} to exchange the order of the limit and limsup, we get 
\begin{multline*}
    0 \leq \lim_{\epsilon \downarrow 0} \varlimsup_{T \uparrow \infty} f(\epsilon,T) \leq \varlimsup_{T \uparrow \infty} \lim_{\epsilon \downarrow 0}  f(\epsilon,T) \\    
    = \varlimsup_{T \uparrow \infty} \frac{1}{T} \E\left[ \int_0^T h_{x}(X^{\xi^\star}_t) (X^{\xi}_t - X^{\xi^\star}_t) dt +K_+(\xi^{+}_T - \xi^{\star,+}_T) +K_-(\xi^{-}_T - \xi^{\star,-}_T) \right] . 
\end{multline*}
This concludes the proof.
\end{proof}

We are ready to prove Theorem \ref{thm:sufficient_condition}.
\begin{proof}[Proof of Theorem \ref{thm:sufficient_condition}]
Let $\xi^\star$ be as in the statement of the theorem.
Let $\xi \in \cB_{mfc}$ be any other admissible control, satisfying the integrability condition \eqref{eq:integrability_control}.
By \eqref{eq:sufficient_condition:assumption_limit}, relying on the equality $\lim_{n}a_n - \varlimsup_{n}b_n = \varliminf_{n}(a_n-b_n)$ and on the joint convexity of $(x,\theta) \mapsto l(x,\theta)$, we get
\begin{equation}\label{eq:thm:sufficient_condition:bound1}
\begin{aligned}
    & \bfJ^{mfc} (\xi^{\star}) - \bfJ^{mfc}(\xi) \leq \varliminf_{T \to \infty} \frac{1}{T}\E\bigg[ \int_0^T \big( l(X^{\xi^{\star}}_t,\theta^{\xi^{\star}}) - l(X^{\xi}_t,\theta^\xi) \big) + K_+ (\xi^{\star,+}_T - \xi^{+}_T) + K_- (\xi^{\star,-}_T - \xi^{-}_T)   \bigg] \\
    & \leq \varliminf_{T \to \infty} \frac{1}{T}\E\bigg[ \int_0^T \big( l_x(X^{\xi^{\star}}_t,\theta^{\xi^{\star}})(X^{\xi^{\star}}_t - X^{\xi}_t) + l_\theta(X^{\xi^{\star}}_t,\theta^{\xi^{\star}})(\theta^{\xi^{\star}} - \theta^{\xi}) \big) + K_+ (\xi^{\star,+}_T - \xi^{+}_T) + K_- (\xi^{\star,-}_T - \xi^{-}_T)   \bigg] \\
    & \leq \varliminf_{T \to \infty} \frac{1}{T}\E\left[ \int_0^T l_x(X^{\xi^{\star}}_t,\theta^{\xi^{\star}})(X^{\xi^{\star}}_t - X^{\xi}_t)  + K_+ (\xi^{\star,+}_T - \xi^{+}_T) + K_- (\xi^{\star,-}_T - \xi^{-}_T)   \right] \\
    & \;\; + \varlimsup_{T \uparrow \infty}\frac{1}{T}\E\left[\int_0^T l_\theta(X^{\xi^{\star}}_t,\theta^{\xi^{\star}}) dt \right](\theta^{\xi^{\star}} - \theta^{\xi}).
\end{aligned}
\end{equation}
By exploiting \eqref{mfg:cons:lambda}, we have
\begin{equation*}
    \lim_{T \to \infty}\frac{1}{T}\E\left[ \int_0^T l_\theta(X^{\xi^{\star}}_t,\theta^{\xi^{\star}})dt \right] = -\lambda^\star.
\end{equation*}
As $\xi^{\star}$ and $\xi$ belong to $\cB^{mfc}$, the processes $X^{\xi^{\star}}$ and $X^{\xi}$ are ergodic, so that we have
\begin{equation}\label{eq:cons_lambda:implication}
    \varlimsup_{T \uparrow \infty}\frac{1}{T}\E\left[\int_0^T l_\theta(X^{\xi^{\star}}_t,\theta^{\xi^{\star}}) dt \right](\theta^{\xi^{\star}} - \theta^{\xi}) = -\lambda^\star \lim_{T \to \infty}\frac{1}{T}\E\left[\int_0^T (X^{\xi^{\star}}_t-X^{\xi}) dt \right].
\end{equation}
By plugging \eqref{eq:cons_lambda:implication} in the last line of \eqref{eq:thm:sufficient_condition:bound1}, we deduce
\begin{equation}\label{eq:thm:sufficient_condition:bound2}
\begin{aligned}
    & \bfJ^{mfc}(\xi^{\star}) - \bfJ^{mfc}(\xi) \leq \varliminf_{T \to \infty} \frac{1}{T}\E\left[ \int_0^T l_x(X^{\xi^{\star}}_t,\theta^{\xi^{\star}})(X^{\xi^{\star}}_t - X^{\xi}_t)  + K_+ (\xi^{\star,+}_T - \xi^{+}_T) + K_- (\xi^{\star,-}_T - \xi^{-}_T)   \right] \\
    & \;\; + \lim_{T \to \infty}\frac{1}{T}\E\left[\int_0^T -\lambda^\star (X^{\xi^{\star}}_t-X^{\xi}) dt \right] \\
    & \leq \varliminf_{T \to \infty} \frac{1}{T}\E\left[ \int_0^T \big( l_x(X^{\xi^{\star}}_t,\theta^{\xi^{\star}}) - \lambda^\star\big) (X^{\xi^{\star}}_t - X^{\xi}_t)  + K_+ (\xi^{\star,+}_T - \xi^{+}_T) + K_- (\xi^{\star,-}_T - \xi^{-}_T)   \right].
\end{aligned}
\end{equation}
By \eqref{eq:optimality:first_order_condition}, the right-hand side of \eqref{eq:thm:sufficient_condition:bound2} is always non-positive, which implies that $\xi^{\star}$ is optimal.
\end{proof}

{\color{black} 
\begin{remark}
The proof of Theorem \ref{thm:sufficient_condition} relies essentially on two structural properties: the affine dependence of the controlled state on the singular control and the joint convexity of the instantaneous cost in the state and interaction variables.
The former allows one to derive the first-order optimality conditions for the frozen-parameter control problem, while the latter turns these conditions into a global comparison with arbitrary admissible controls.
Therefore, the same verification argument should extend to affine controlled dynamics and jointly convex instantaneous costs.
\end{remark}
}

Finally, we prove that the potential stationary MFG admits a solution for which \eqref{eq:sufficient_condition:assumption_limit} and \eqref{eq:integrability_control} hold.
By Theorem \ref{thm:sufficient_condition}, this yields the optimality of the control $\xi^\star$ for the stationary MFC problem {\color{black} within the class of admissible controls satisfying the additional integrability requirement \eqref{eq:integrability_control}}.

\begin{proof}[Proof of Theorem \ref{thm:mfg:existence}]
For any {\color{black} $(\theta,\lambda) \in \R^2$}, let $\xi^{\star}(\theta,\lambda)$ be given by Proposition \ref{prop:ergodic_ctrl:solution} and recall that it {\color{black} minimizes the cost} $\bfJ(\xi,\theta,\lambda)$ over $\xi \in \cB$.
Moreover, recall that the diffusion $X^{\xi^{\star}(\theta,\lambda)}$ admits a unique stationary distribution $p^{\theta,\lambda}_\infty \in \cP(\R)$, given by \eqref{eq:stationary_distribution}.
Define the map $G : \R^2 \to \R^2$
\begin{equation}\label{eq:map_fixed_point}
    G(\theta,\lambda) = \left( \int_{\R} x \, p^{\theta,\lambda}_\infty(dx), - \int_{\R} l_\theta(x,\theta) \, p^{\theta,\lambda}_\infty(dx) \right),
\end{equation}
and denote by $G_1$ and $G_2$ its first and second components.
Notice that $(\xi^{\star}(\theta^\star,\lambda),\theta^\star,\lambda^\star)$ is an equilibrium for the potential stationary MFG if and only if $(\theta^\star,\lambda^\star)$ is {\color{black} fixed point} of the map $G$.
Thus, we show that there exists a {\color{black} fixed point} of the map $G$.

\smallskip
Define the map 
\begin{equation}\label{eq:lambda_of_theta}
    \lambda^\star(\theta) \coloneqq \psi(\bar{\theta} - \theta) 
\end{equation}
and consider $G_1(\theta,\lambda^\star(\theta))$.
We search for a {\color{black} fixed point of $G_1(\theta,\lambda^\star(\theta))$}.
Indeed, suppose that, $\theta^\star$ is a {\color{black} fixed point} of $G_1(\theta,\lambda^\star(\theta))$ and set $\lambda^\star \coloneqq \lambda^\star(\theta^\star)$.
Then, the pair $(\theta^\star,\lambda^\star)$ is a {\color{black} fixed point} of $G$, as it holds
\begin{align*}
    G_2(\theta^\star,\lambda^\star) & = -\int_{\R} l_\theta(x,\theta^\star)p^{\theta^\star,\lambda^\star}_\infty(dx) \\
    & = -\int_{\R} \left( -\phi(x-\theta^\star) + \psi(\theta^\star - \bar{\theta}) \right)  \, p^{\theta^\star,\lambda^\star}_\infty(dx) = \lambda^\star(\theta^\star).
\end{align*}
We turn our attention to finding a {\color{black} fixed point} of $G_1(\theta,\lambda^\star(\theta))$.
Exploiting $\P ( a_+(\theta,\lambda) \leq X^{\xi^{\star}(\theta,\lambda)}_t \leq a_-(\theta,\lambda) ) = 1$ for any $t \geq 0$, we get
\begin{equation*}
    a_+(\theta,\lambda) \leq \varlimsup_{T \uparrow \infty} \frac{1}{T} \int_0^T \E[X^{\xi^{\star}(\theta,\lambda)}_t] dt \leq a_-(\theta,\lambda).
\end{equation*}
Exploiting the ergodicity of $X^{\xi^{\star}(\theta,\lambda^\star(\theta))}$, we deduce
\begin{equation}\label{eq:fixed_point:linear_bounds1}
    a_+(\theta,\lambda^\star(\theta)) \leq G_1(\theta,\lambda^\star(\theta)) \leq a_-(\theta,\lambda^\star(\theta)).
\end{equation}
We repeat the same reasoning as in the proof of Proposition \ref{prop:mfg:fixed_multiplier}:
We start by noticing that the map $\theta \mapsto G_1(\theta,\lambda^\star(\theta))$ is continuous, since $\lambda^\star(\theta)$ is continuous and so are the maps $a_\pm(\theta,\lambda)$.
Thus, it is enough to show that there exist $\theta_-$ such that $G_1(\theta_-,\lambda^\star(\theta_-)) - \theta_- < 0$ and $\theta_+$ such that $G_1(\theta_+,\lambda^\star(\theta_+)) - \theta_+ > 0$, with $\theta_+ < \theta_-$.
By continuity, this implies the existence of $\theta^\star \in [\theta_+,\theta_-]$ such that $G_1(\theta^\star,\lambda^\star(\theta^\star)) = \theta^\star$.
To do so, we exploit the bounds in \eqref{eq:fixed_point:bounds} and the differentiability of $\lambda^\star(\theta)$ to show that there exist $\theta_-$ such that $a_-(\theta_-,\lambda^\star(\theta_-)) - \theta_- < 0$ and $\theta_+$ such that $a_+(\theta_+,\lambda^\star(\theta_+)) - \theta_+ > 0$.
Since $\lambda^\star(\theta)=\psi(\bar\theta-\theta)$, at points where $a_\pm$ is differentiable we have
\[
\frac{d}{d\theta}a_\pm(\theta,\lambda^\star(\theta))
=\partial_\theta a_\pm(\theta,\lambda^\star(\theta))-\psi\,\partial_\lambda a_\pm(\theta,\lambda^\star(\theta))
\le 2m-c\psi.
\]
Take now $\theta < 0$. We have
\begin{multline*}
    a_+(\theta,\lambda^\star(\theta)) =  a_+(0,\lambda^\star(0)) + \int_0^\theta \frac{d a_+(\theta',\lambda^\star(\theta'))}{d\theta} d\theta' = a_+(0,\lambda^\star(0)) - \int_\theta^0 \frac{d a_+(\theta',\lambda^\star(\theta'))}{d\theta} d\theta' \\
    \geq a_+(0,\lambda^\star(0)) + (2m - c\psi)\theta.
\end{multline*}
This implies
\[
a_+(\theta,\lambda^\star(\theta)) - \theta \geq (2m-1 - c\psi)\theta + a_+(0,\lambda^\star(0)) > 0.
\]
Since $\phi < \rho$ by Assumption \ref{ass:phi_rho}, we have $2m < 1$. Thus, the previous inequality is satisfied by $\theta_+ \leq \frac{a_+(0,\lambda^\star(0))}{1 - 2m + c\psi} \wedge 0$.
Analogously for $a_-(\theta,\lambda^\star(\theta))$, take $\theta > 0$ and consider the upper bound
\[
    a_-(\theta,\lambda^\star(\theta)) =  a_-(0,\lambda^\star(0)) + \int_0^\theta \frac{d a_-(\theta',\lambda^\star(\theta'))}{d\theta} d\theta' \leq a_-(0,\lambda^\star(0)) + (2m - c\psi)\theta.
\]
so that $a_-(\theta,\lambda) - \theta < 0$ if $(1 - 2m +c\psi )\theta > a_-(0,\lambda^\star(0))$.
Since $2m < 1$, the previous inequality is satisfied by $\theta_- > \frac{a_-(0,\lambda^\star(0))}{1 - 2m + c\psi} \vee 0$.

Finally, we notice that $\xi^\star$ satisfies \eqref{eq:sufficient_condition:assumption_limit} by \eqref{eq:final_ctrl} and that condition \eqref{eq:integrability_control} is verified, as $X^{\xi^\star}$ is bounded by two constants, being the solution of the Skorohod reflection problem \eqref{eq:skorohod_reflection}.
\end{proof}

{\color{black} We next show that the solution to the potential stationary MFG is unique within the class of reflection policies, which also implies uniqueness of the stationary MFC optimizer within this class.

\begin{proposition}\label{prop:uniqueness_reflected_mfg}
The potential stationary MFG admits at most one solution in the class of reflection policies.
Consequently, the optimal policy for the stationary MFC problem constructed in Theorem \ref{thm:mfg:existence} is the unique optimizer in the class of reflection policies.
\end{proposition}

\begin{proof}
Let $(\xi^1,\theta_1,\lambda_1)$ and $(\xi^2,\theta_2,\lambda_2)$ be two solutions to the potential stationary MFG in the class of reflection policies.
For $k=1,2$, write $X^k:=X^{\xi^k}$, and set
\[
Q:=\varliminf_{T\to\infty}\frac{1}{T}\E\left[\int_0^T |X^1_t-X^2_t|^2dt\right].
\]
For $\epsilon\in(0,1)$, define $\xi^{\epsilon,\pm}_t = (1-\epsilon)\xi^{1,\pm}+\epsilon\xi^{2,\pm}$ and set $\xi^\epsilon_t = \xi^{\epsilon,+}_t - \xi^{\epsilon,-}_t$.
By linearity of the dynamics, $X^{\xi^\epsilon}=(1-\epsilon)X^1+\epsilon X^2$.
Moreover, since $x\mapsto l(x,\theta)$ is quadratic with second order coefficient $\rho+\phi$, explicit computations yield
\begin{equation}\label{eq:prop_uniqueness:quadratic_expansion}
    l(X^{\xi^\epsilon}_t,\theta)=(1-\epsilon)l(X^1_t,\theta)+\epsilon l(X^2_t,\theta)-\frac{\rho+\phi}{2}\epsilon(1-\epsilon)|X^1_t-X^2_t|^2,
\end{equation}
for any $\theta \in \R$.
Since $\xi^1$ minimizes $\bfJ(\cdot,\theta_1,\lambda_1)$, exploiting \eqref{eq:prop_uniqueness:quadratic_expansion}, the linearity of $x \mapsto \lambda(\theta - x)$ and the linearity of the control terms, we have
\[
\bfJ(\xi^1,\theta_1,\lambda_1) \leq \bfJ(\xi^\epsilon,\theta_1,\lambda_1) \leq (1-\epsilon)\bfJ(\xi^1,\theta_1,\lambda_1) + \epsilon \bfJ(\xi^2,\theta_1,\lambda_1) - \epsilon(1-\epsilon)\frac{\phi + \rho}{2}Q,
\]
which, by rearranging the terms, dividing by $\epsilon$ and taking the limit as $\epsilon \to 0$, yields
\[
\bfJ(\xi^2,\theta_1,\lambda_1)-\bfJ(\xi^1,\theta_1,\lambda_1)\geq \frac{\rho+\phi}{2}Q.
\]
Interchanging the roles of the two solutions yields
\[
\bfJ(\xi^1,\theta_2,\lambda_2)-\bfJ(\xi^2,\theta_2,\lambda_2) \geq \frac{\rho+\phi}{2}Q.
\]
Adding these inequalities gives
\begin{equation}\label{eq:prop_uniqueness:inequality_costs}
    \bfJ(\xi^2,\theta_1,\lambda_1) - \bfJ(\xi^1,\theta_1,\lambda_1) + \bfJ(\xi^1,\theta_2,\lambda_2) - \bfJ(\xi^2,\theta_2,\lambda_2) \geq (\rho+\phi)Q.
\end{equation}
We recall that, since $\xi^1$ and $\xi^2$ are reflection policies, \cite[Lemma 2.1]{alvarez2018stationary} implies that all the limsups appearing in the definition of $\bfJ(\xi^j,\theta_k,\lambda_k)$, $k,j = 1,2$, are ordinary limits.
Explicit manipulations in \eqref{eq:prop_uniqueness:inequality_costs} yield then the inequality
\begin{multline*}
    (\rho + \phi)Q \leq \lim_{T \to \infty} \frac{1}{T} \E \left[ \int_0^T \left( \phi (\theta_2 -\theta_1) + \lambda_2 - \lambda_1 \right)(X^2_t - X^1_t) \right] \\
    = \lim_{T \to \infty} \frac{1}{T} \E \left[ \int_0^T (\phi - \psi) (\theta_2 -\theta_1) (X^2_t - X^1_t) \right] = (\phi - \psi) \vert \theta^2- \theta_1\vert^2,
\end{multline*}
where we exploited the fact that $\lambda_k=\psi(\bar\theta-\theta_k)$ by \eqref{mfg:cons:lambda} in the first equality and the consistency condition \eqref{mfg:cons:theta} in the second equality.
On the other hand, Jensen's inequality and the consistency condition \eqref{mfg:cons:theta} imply
\[
Q = \varliminf_{T \to \infty} \frac{1}{T} \E\left[\int_0^T \vert X^1_t - X^2_t \vert^2 dt \right] \geq \varliminf_{T \to \infty} \left( \frac{1}{T} \E\left[\int_0^T (X^1_t - X^2_t )dt \right] \right)^2 =   |\theta_1-\theta_2|^2.
\]
Consequently,
\[
(\phi-\psi)|\theta_1-\theta_2|^2\geq(\rho+\phi)|\theta_1-\theta_2|^2,
\]
and thus $(\rho+\psi)|\theta_1-\theta_2|^2 \leq 0$.
Since $\rho>0$ and $\psi\geq0$, it follows that $\theta_1=\theta_2$, and therefore also $\lambda_1=\lambda_2$.
For fixed interaction parameters, the optimal reflection boundaries are unique, and the associated Skorokhod reflection problem admits a unique solution.
Hence, $\xi^1=\xi^2$.
Finally, if two reflection policies were both optimal for the stationary MFC problem, Theorem \ref{thm:necessary_condition} would imply that both induce solutions to the potential stationary MFG.
By the first part of the proof, they must coincide.
\end{proof}
}

{\color{black}
We conclude this section by relating the stationary formulation in \eqref{eq:cost:ergodic} to the time-dependent mean-field criterion introduced in Remark \ref{rmk:finite_moments_problem}.
We show that the optimal reflection policy constructed above is also optimal for the time-dependent MFC problem, within the class of controls satisfying \eqref{eq:integrability_control}.

\begin{proposition}\label{prop:optimality_time_dependent_mfc}
Let $(\xi^\star,\theta^\star,\lambda^\star)$ be the solution to the potential stationary MFG constructed in Theorem \ref{thm:mfg:existence}.
Then,
\[ \widehat{\bfJ}^{mfc}(\xi^\star)\leq \widehat{\bfJ}^{mfc}(\xi), \]
for every $\xi\in\mathcal B$ satisfying \eqref{eq:integrability_control}.
\end{proposition}

\begin{proof}
Set $\theta_t^\star:=\E[X_t^{\xi^\star}]$.
Since $\xi^\star$ reflects $X^{\xi^\star}$ between two finite boundaries and $p_\infty^\star$ is its limiting stationary distribution, boundedness of $X^{\xi^\star}$ implies that $\theta_t^\star\to\theta^\star$ as $t\to\infty$.
Hence,
\begin{equation}\label{eq:convergence_mean}
    \lim_{T\to\infty}\frac{1}{T}\int_0^T|\theta_t^\star-\theta^\star|dt=\lim_{T\to\infty}\frac{1}{T}\int_0^T|\theta_t^\star-\theta^\star|^2dt=0.
\end{equation}
Denote by $\bfJ_T^{mfc}(\xi^\star)$ and $\widehat{\bfJ}^{mfc}_T(\xi^\star)$ the finite-horizon averages corresponding to $\bfJ^{mfc}(\xi^\star)$ and $\widehat{\bfJ}^{mfc}(\xi^\star)$.
Expanding the squares, we obtain
\[
\lim_{T \to \infty} \left( \widehat{\bfJ}^{mfc}_T(\xi^\star)-\bfJ_T^{mfc}(\xi^\star) \right) =\lim_{T \to \infty} \left( \frac{\psi-\phi}{2T}\int_0^T(\theta_t^\star-\theta^\star)^2dt+\frac{\psi(\theta^\star-\bar\theta)}{T}\int_0^T(\theta_t^\star-\theta^\star)dt \right) = 0.
\]
Moreover, \eqref{eq:final_ctrl} yields $\bfJ^{mfc}(\xi^\star) = \lim_{T \to \infty} \bfJ_T^{mfc}(\xi^\star)$.
Therefore,
\[
\widehat{\bfJ}^{mfc}(\xi^\star) = \varlimsup_{T \to \infty} \widehat{\bfJ}^{mfc}_T (\xi^\star) = \varlimsup_{T\to\infty} \left( \widehat{\bfJ}^{mfc}_T (\xi^\star)- \bfJ^{mfc}_T (\xi^\star) \right) + \lim_{T \to \infty} \bfJ^{mfc}_T (\xi^\star) =\bfJ^{mfc}(\xi^\star).
\]
Fix now $\xi\in\mathcal B$ satisfying \eqref{eq:integrability_control}, and set $\theta_t^\xi:=\E[X_t^\xi]$.
By joint convexity of $l(x,\theta)$, using $\widehat{\bfJ}^{mfc}(\xi)=\varlimsup_{T\to\infty}\widehat{\bfJ}^{mfc}_T(\xi)$ and $\bfJ^{mfc}(\xi^\star)=\lim_{T\to\infty}\bfJ_T^{mfc}(\xi^\star)$, we have
\begin{multline*}
\bfJ^{mfc}(\xi^\star)-\widehat{\bfJ}^{mfc}(\xi)
\leq \varliminf_{T\to\infty}\frac{1}{T}\Bigg(\E\left[\int_0^T l_x(X_t^{\xi^\star},\theta^\star)(X_t^{\xi^\star}-X_t^\xi)dt+K_+(\xi_T^{\star,+}-\xi_T^+)+K_-(\xi_T^{\star,-}-\xi_T^-)\right] \\
+\int_0^T\E\left[l_\theta(X_t^{\xi^\star},\theta^\star)\right](\theta^\star-\theta_t^\xi)dt\Bigg).
\end{multline*}
Since $l_\theta(x,\theta)=\phi(\theta-x)+\psi(\theta-\bar\theta)$ and $\lambda^\star=\psi(\bar\theta-\theta^\star)$, it holds $\E[l_\theta(X_t^{\xi^\star},\theta^\star)]=\phi(\theta^\star-\theta_t^\star)-\lambda^\star$.
Set
\[
r_T:=\frac{1}{T}\int_0^T(\theta^\star-\theta_t^\star)\left(\phi(\theta^\star-\theta_t^\xi)-\lambda^\star\right)dt.
\]
By using $\theta_t^\xi-\theta_t^\star=\E[X_t^\xi-X_t^{\xi^\star}]$, we have
\[
\frac{1}{T}\int_0^T\E\left[l_\theta(X_t^{\xi^\star},\theta^\star)\right](\theta^\star-\theta_t^\xi)dt=\frac{\lambda^\star}{T}\int_0^T(\theta_t^\xi-\theta_t^\star)dt+r_T = -\frac{\lambda^\star}{T}\E\left[\int_0^T(X_t^{\xi^\star}-X_t^\xi)dt\right] + r_T. 
\]
Hence,
\begin{multline*}
    \bfJ^{mfc}(\xi^\star)-\widehat{\bfJ}^{mfc}(\xi) \\
    \leq\varliminf_{T\to\infty}\left(\frac{1}{T}\E\left[\int_0^T\left(l_x(X_t^{\xi^\star},\theta^\star)-\lambda^\star\right)(X_t^{\xi^\star}-X_t^\xi)dt+K_+(\xi_T^{\star,+}-\xi_T^+)+K_-(\xi_T^{\star,-}-\xi_T^-)\right]+r_T\right).
\end{multline*}
By Cauchy-Schwarz and Jensen's inequality and \eqref{eq:integrability_control}, we have
\[
\lim_{T \to \infty} \vert r_T\vert \leq \lim_{T \to \infty} \phi\left(\frac{1}{T}\int_0^T\vert\theta^\star-\theta_t^\star\vert^2dt\right)^{1/2}\left(\frac{1}{T}\int_0^T\vert\theta^\star-\theta_t^\xi\vert^2dt\right)^{1/2} + \lim_{T \to \infty} \vert\lambda^\star\vert\frac{1}{T}\int_0^T\vert\theta^\star-\theta_t^\star\vert dt = 0,
\]
by relying on \eqref{eq:convergence_mean} and on \eqref{eq:integrability_control} to deduce that
\[
\sup_{T>0}\frac{1}{T}\int_0^T\vert\theta^\star-\theta_t^\xi\vert^2dt\leq 2\vert\theta^\star\vert^2+2\sup_{T>0}\frac{1}{T}\E\left[\int_0^T\vert X_t^\xi\vert^2dt\right]<\infty.
\]
Consequently,
\[
\bfJ^{mfc}(\xi^\star)-\widehat{\bfJ}^{mfc}(\xi) \leq \varliminf_{T\to\infty} \frac{1}{T}\E\left[\int_0^T\left(l_x(X_t^{\xi^\star},\theta^\star)-\lambda^\star\right)(X_t^{\xi^\star}-X_t^\xi)dt+K_+(\xi_T^{\star,+}-\xi_T^+)+K_-(\xi_T^{\star,-}-\xi_T^-)\right],
\]
which is always non-positive thanks to Proposition \ref{prop:first_order}.
The claim follows from $\widehat{\bfJ}^{mfc}(\xi^\star)=\bfJ^{mfc}(\xi^\star)$.
\end{proof}

}

\section{Application to optimal inventory management: numerical illustrations}\label{sec:numerics}

{\color{black} 
In this section, we revisit the continuous-time optimal inventory-management interpretation introduced in the Introduction and illustrate numerically the optimal policy for the stationary MFC problem.
Recall that $X^\xi$ represents the net inventory position of a firm: positive values correspond to available inventory, whereas negative values represent unmet customer demand accumulated as backlog.
The upward control $\xi^+$ represents exceptional replenishment actions undertaken to increase available inventory or reduce backlog, whereas the downward control $\xi^-$ represents exceptional actions aimed at reducing excess inventory, for instance through sales, transfers, or disposal of stored goods (see the discussion in \cite{federico2023inventory_unknown_demand}).
The cost functional combines ingredients appearing in related continuous-time inventory models, such as a quadratic penalty on deviations from a desired inventory position $\bar{x}$ (see, e.g., \cite{cadenillas2010optimal}).
Its distinctive feature is the dependence on the stationary mean induced by the implemented policy, through the penalization of the long-term variance of the net inventory position and of deviations of its long-run average from a prescribed target $\bar{\theta}$.
According to our theoretical results, the optimal policy is of reflection type: the firm maintains its net inventory position, with minimal intervention, inside an optimal band.
For notational convenience, we denote the lower and upper boundaries of the optimal band by $a^\star_\pm:=a_\pm(\theta^\star,\lambda^\star)$.
We investigate how the optimal band $[a^\star_+,a^\star_-]$ and the associated optimal long-run average inventory $\theta^\star$ depend on the speed of mean reversion $\delta$, the volatility $\sigma$, and the weight $\phi$ assigned to the long-term variance of the net inventory position.
}

\subsection*{Numerics}
The first step is to solve for $a_\pm(\theta,\lambda)$, for any $(\theta,\lambda) \in \R^2$. To do so, we need to solve the one-dimensional free-boundary problem
\begin{equation}\label{eq:free_boundary:zero}
    \left\{ \begin{aligned}
        -\delta U(x,\theta,\lambda) -\delta x U_{x}(x,\theta,\lambda) + \frac{1}{2}\sigma^2 U_{xx}(x,\theta,\lambda) + (\rho + \phi)x - \rho\bar{x} -\phi\theta - \lambda = 0, \quad & a_+(\theta,\lambda) < x < a_-(\theta,\lambda), \\
        U(x,\theta,\lambda) = -K_+, \quad & x \leq a_+(\theta,\lambda), \\
        U(x,\theta,\lambda) = K_-, \quad & x \geq a_-(\theta,\lambda),
    \end{aligned} \right.
\end{equation}
numerically. We remark that, by Lemma \ref{lemma:Dynkin_game:solution}, a solution of the free-boundary problem exists.

\begin{lemma}\label{numerics:lemma:solutions}
The two fundamental solutions of the equation
\[
-\delta U(x) + \cL_{X^0}U(x) = 0
\]
are given by
\begin{equation}
\begin{aligned}
    {\color{black} \Phi(x)} = e^{\frac{\alpha x^2}{2}} F\Big ( \sqrt{\frac{\alpha}{2}} x \Big), \quad {\color{black} \Psi(x)} = e^{\frac{\alpha x^2}{2}} F\Big( - \sqrt{\frac{\alpha}{2}} x \Big)
\end{aligned}
\end{equation}
where
\[
\alpha = \frac{2\delta}{\sigma^2}, \quad F(x) = \operatorname{Erfc}(x) = \frac{2}{\sqrt{\pi}}\int_{x}^\infty e^{-t^2}dt.
\]
The particular solution of the equation
\begin{equation}\label{eq:numerics:ODE}
    -\delta U(x) + \cL_{X^0}U(x) + (\rho + \phi)x - \rho\bar{x} -\phi\theta -\lambda = 0
\end{equation}
is given by
\begin{equation}
    \bar{U}(x) = \frac{\rho + \phi}{2\delta}x - \frac{\rho \bar{x} + \phi\theta +\lambda}{\delta}.
\end{equation}
\end{lemma}
The proof of Lemma \ref{numerics:lemma:solutions} follows from straightforward computations and it is therefore omitted.
Thanks to Lemma \ref{numerics:lemma:solutions}, any solution to \eqref{eq:numerics:ODE} is given by $A \Phi(x) + B \Psi(x) + \bar{U}(x)$, with $A$ and $B$ real constants.
Thus, we only need to find such constants by imposing the smooth fit condition $U \in \dC^1(\R)$ in \eqref{eq:free_boundary:zero} at the points $a_\pm(\theta,\lambda)$.
This gives a pair of equations that uniquely determine the free-boundaries $a_\pm(\theta,\lambda)$.

\begin{lemma}
The free-boundaries $a_\pm(\theta,\lambda)$ are given by the unique solution of the following system of equations:
\begin{equation}
\left\{ \begin{aligned}
    & \frac{\Psi(a_+)\, \bar{U}_x(a_+) - \Psi_x(a_+)\, \big(K_+ + \bar{U}(a_+)\big)}{\Phi(a_+)\, \Psi_x(a_+) - \Psi(a_+)\, \Phi_x(a_+)} = \frac{\Psi_x(a_-)\, \big(K_- - \bar{U}(a_-)\big) + \Psi(a_-)\, \bar{U}_x(a_-)}{\Phi(a_-)\, \Psi_x(a_-) - \Psi(a_-)\, \Phi_x(a_-)}, \\
    & \frac{\Phi_x(a_+)\, \big(K_+ + \bar{U}(a_+)\big) - \Phi(a_+)\, \bar{U}_x(a_+)}{\Phi(a_+)\, \Psi_x(a_+) - \Psi(a_+)\, \Phi_x(a_+)} = \frac{-\Phi_x(a_-)\, \big(K_- - \bar{U}(a_-)\big) - \Phi(a_-)\, \bar{U}_x(a_-)}{\Phi(a_-)\, \Psi_x(a_-) - \Psi(a_-)\, \Phi_x(a_-)}.
\end{aligned} \right. 
\end{equation}
\end{lemma}
The proof follows from imposing the smooth-fit conditions $U(a_\pm) = \mp K_\pm $ and $U_x(a_\pm) = 0$ at $a_\pm(\theta,\lambda)$, and it is therefore omitted.

\smallskip
The next step is to find the equilibrium parameters $(\theta^\star,\lambda^\star)$.
By the proof of Theorem \ref{thm:mfg:existence}, $\lambda^\star = \psi(\bar{\theta} - \theta^\star)$.
Hence, it is enough to find numerically the {\color{black} fixed point} of the map
\begin{equation}\label{eq:numerics:fixed_point}
    G(\theta) = \left( \int_{a_+(\theta,\psi(\bar{\theta} - \theta)) }^{ a_-(\theta,\psi(\bar{\theta} - \theta))  } e^{-\frac{\delta}{\sigma^2} y^2 } dy  \right)^{-1} \int_{a_+(\theta,\psi(\bar{\theta} - \theta)) }^{ a_-(\theta,\psi(\bar{\theta} - \theta))  } x e^{-\frac{\delta}{\sigma^2} x^2 } dx.
\end{equation}

\subsection*{Illustrations}

{\color{black} 
Figure \ref{fig:delta} shows how the lower and upper boundaries $a^\star_+$ and $a^\star_-$ of the optimal band, as well as the optimal long-run average inventory $\theta^\star$, depend on the speed of mean-reversion $\delta$.
In the inventory interpretation, the parameter $\delta$ measures the speed at which the uncontrolled net inventory position is pulled toward zero: available inventory is depleted when it is positive, whereas backlog is gradually reduced when the inventory position is negative.
A larger value of $\delta$ therefore strengthens this adjustment mechanism and reduces the need for costly interventions.
Consequently, the firm can tolerate a wider range of inventory positions before intervening: the lower boundary $a^\star_+$ decreases, the upper boundary $a^\star_-$ increases, and the optimal band widens.
This widening of the optimal band as the mean-reversion speed increases is qualitatively consistent with the literature (see, e.g., the numerical findings of \cite{cadenillas2010optimal} for a related mean-reverting inventory model with two-sided impulse adjustments).
The right panel shows that $\theta^\star$ moves toward zero as $\delta$ increases.
This is consistent with the fact that stronger mean reversion pulls the long-run average inventory more forcefully toward the natural balance level zero of the uncontrolled Ornstein-Uhlenbeck dynamics.

Figure \ref{fig:sigma} illustrates the sensitivity of the optimal policy with respect to the volatility coefficient $\sigma$.
A larger value of $\sigma$ corresponds to greater uncertainty affecting regular supply and customer demand and, therefore, to more pronounced fluctuations of the net inventory position. Since correcting each fluctuation through exceptional replenishment or inventory reduction would entail proportional adjustment costs, the firm optimally tolerates a wider range of inventory positions before intervening. 
Accordingly, the lower boundary $a^\star_+$ decreases and the upper boundary $a^\star_-$ increases, so that the optimal band expands. This widening of the optimal band as volatility increases is qualitatively consistent with the literature (see again \cite{cadenillas2010optimal}).
The right panel shows that, for the parameter configuration considered here, the optimal long-run average inventory $\theta^\star$ increases with $\sigma$ and appears to approach an asymptotic level as volatility becomes large.
Since both reference levels $\bar{x}$ and $\bar{\theta}$ are positive in this illustration, this shows that the optimal response to higher uncertainty consists not only in widening the optimal band, but also in selecting a higher long-run average inventory position.
Moreover, the stabilization of $\theta^\star$ for large values of $\sigma$ indicates that further increases in uncertainty are accommodated mainly through greater dispersion of the optimally controlled inventory within the wider band, rather than through a substantial further change in its long-run average level.

Figure~\ref{fig:phi} illustrates the sensitivity of the optimal inventory policy with respect to the parameter $\phi$, under the assumption $\phi<\rho$ of Assumption~\ref{ass:phi_rho}.
Unlike $\delta$ and $\sigma$, which describe the inventory dynamics, the parameter $\phi$ governs the long-term variance component of the mean-field dependence in the cost functional.
Indeed, since $\theta^\star$ is the stationary mean induced by the optimal policy, the term weighted by $\phi$ penalizes the long-term variance of the net inventory position.
Hence, in addition to keeping the net inventory position close to a desired operational position and its long-run average close to a prescribed target, the firm aims at keeping the long-term variance of its inventory position low.
As $\phi$ increases, deviations of the net inventory position from its long-run average become more costly.
The firm therefore intervenes more tightly: the lower boundary $a^\star_+$ increases, the upper boundary $a^\star_-$ decreases, and the optimal band contracts.
The right panel shows that, for the parameter configuration considered here, the optimal long-run average inventory $\theta^\star$ increases with $\phi$.
Since both reference levels $\bar{x}$ and $\bar{\theta}$ are positive in this illustration, this shows that assigning greater weight to the long-term variance is accompanied by the selection of a higher long-run average inventory position.
}

\begin{figure}[!ht]
    \centering
    \includegraphics[width =0.8\textwidth]{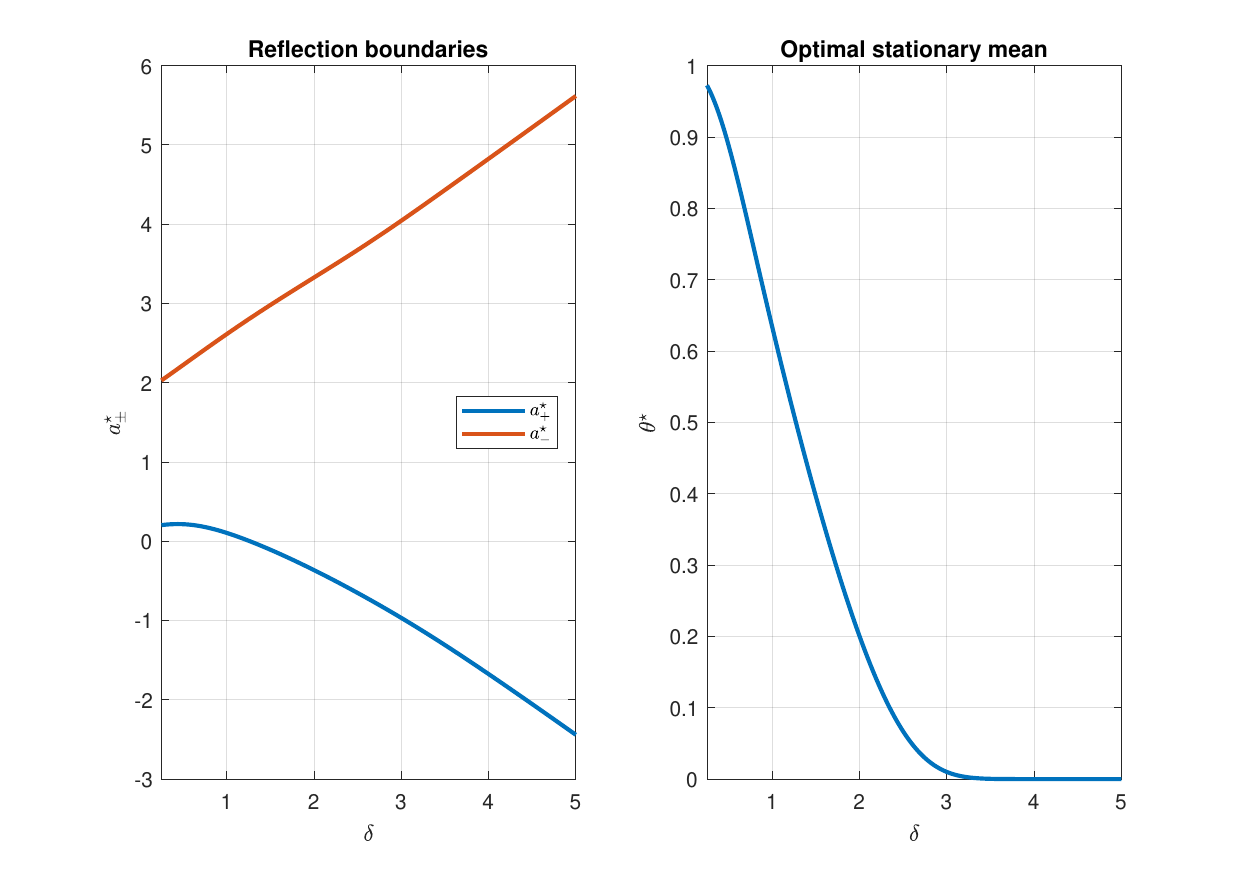}
    \caption{Sensitivity of the optimal inventory policy with respect to the mean-reversion speed $\delta$, varying from $0.1$ to $5$. On the left panel, we show the lower and upper boundaries $a^\star_\pm$ of the optimal band, whereas on the right panel we show the optimal long-run average inventory $\theta^\star$. Here, $\sigma=1$, $\rho=1.5$, $\phi=1$, $\psi=0.5$, $\bar{x}=\bar{\theta}=K_+=K_-=1$.}    \label{fig:delta}
\end{figure}

\begin{figure}[!ht]
    \centering
    \includegraphics[width =0.8\textwidth]{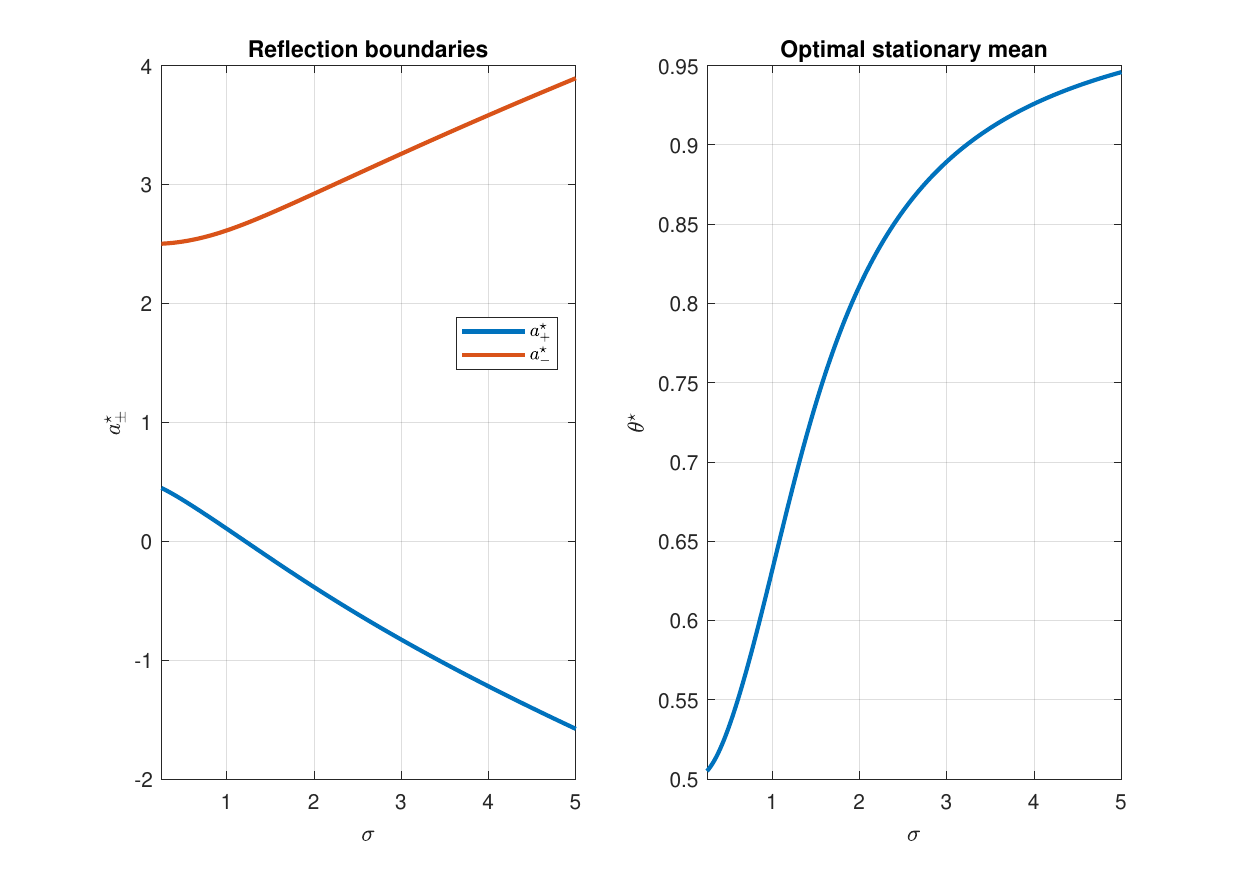}
    \caption{Sensitivity of the optimal inventory policy with respect to the volatility $\sigma$, varying from $0.1$ to $4$. On the left panel, we show the lower and upper boundaries $a^\star_\pm$ of the optimal band, whereas on the right panel we show the optimal long-run average inventory $\theta^\star$. Here, $\delta=1$, $\rho=1.5$, $\phi=1$, $\psi=0.5$, $\bar{x}=\bar{\theta}=K_+=K_-=1$.} \label{fig:sigma}
\end{figure}

\begin{figure}[!ht]
    \centering
    \includegraphics[width =0.8\textwidth]{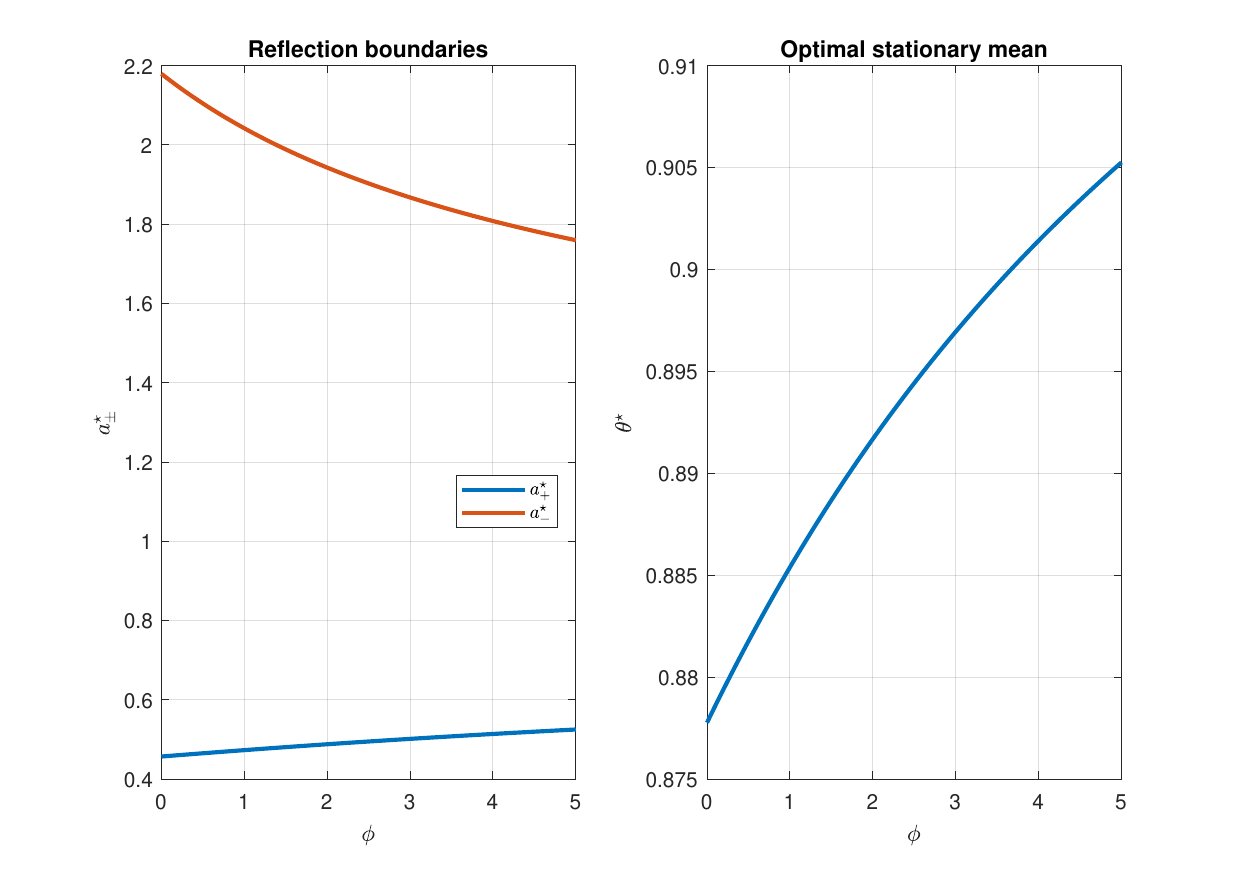}
    \caption{Sensitivity of the optimal inventory policy with respect to the long-term variance weight $\phi$. On the left panel, we show the lower and upper boundaries $a^\star_\pm$ of the optimal band, whereas on the right panel we show the optimal long-run average inventory $\theta^\star$. Here, $\delta=\sigma=1$, $\rho=5$, $\psi=0.5$, $\bar{x}=\bar{\theta}=K_+=K_-=1$.}
    \label{fig:phi}
\end{figure}

\section*{Acknowledgments}
\noindent
Giorgio Ferrari and Umberto Pappalettera are acknowledged by the author for stimulating discussions.
{\color{black} We thank the anonymous referees for their careful reading of the manuscript and their valuable suggestions.}
Funded by the Deutsche Forschungsgemeinschaft (DFG, German Research Foundation) - Project-ID 317210226 - SFB 1283.


\end{document}